\documentclass[12pt,a4paper]{article}
\usepackage[margin=1.2in]{geometry}
\usepackage{times}
\usepackage{xcolor}
\usepackage{amsfonts}
\usepackage{amsmath}
\usepackage{amssymb}
\usepackage{amsthm}
\usepackage{mathtools}
\usepackage{algorithm}
\usepackage{algorithmic}
\usepackage{comment}
\usepackage{hyperref}
\usepackage{soul}
\usepackage{float}
\usepackage{subcaption}
\usepackage{graphicx}
\usepackage{color}
\usepackage{wrapfig}
\newtheorem{assumption} {Assumption}

\newtheorem{theorem} {Theorem}
\newtheorem{lemma} {Lemma}
\newtheorem{definition} {Definition}

\newtheorem{observation} {Observation}

\def\x{{\mathbf{x}}}
\def\e{{\mathbf{e}}}

\def\u{{\mathbf{u}}}
\def\v{{\mathbf{v}}}
\def\z{{\mathbf{z}}}
\def\w{{\mathbf{w}}}
\def\y{{\mathbf{y}}}
\def\p{{\mathbf{p}}}
\def\b{{\mathbf{b}}}
\def\X{{\mathbf{X}}}

\def\A{{\mathbf{A}}}

\def\M{{\mathbf{M}}}
\def\I{{\mathbf{I}}}

\def\U{{\mathbf{U}}}

\def\FO{{\mF^*}}
\def\DFO{D_{\FO}}
\def\DO{{D^*}}
\def\DK{{D_{\mK}}}
\def\xo{{\x^*}}
\def\tnabla{{\tilde{\nabla}}}
\def\tnbt{{\tnabla_t}}
\def\hbar{{\bar{h}}}
\def\dbar{{\bar{d}}}
\def\Mbar{{\bar{M}}}
\def\tilO{{\tilde{O}}}

\def\teta{{\tilde{\eta}}}

\def\ha{{\hat{a}}}

\newcommand{\mP}{\mathcal{P}}
\newcommand{\mV}{\mathcal{V}}
\newcommand{\mF}{\mathcal{F}}

\newcommand{\mK}{\mathcal{K}}
\newcommand{\mX}{\mathcal{X}}

\newcommand{\mS}{\mathcal{S}}

\newcommand{\mbS}{\mathbb{S}}

\newcommand{\E}{\mathbb{E}}

\newcommand{\conv}{\textrm{conv}}

\newcommand{\dist}{\textrm{dist}}

\newcommand{\trace}{\textrm{Tr}}

\newcommand{\reals}{\mathbb{R}}

\DeclareMathOperator*{\argmin}{arg\,min}

\allowdisplaybreaks

\begin{document}
	\date{}
	\author{Dan Garber \\ {\small Technion - Israel Institute of Technology}\\ {\small \texttt{dangar@technion.ac.il}}
		\and
		Noam Wolf\\ {\small Technion - Israel Institute of Technology}\\ {\small \texttt{wolfnoam@campus.technion.ac.il}}}
	\title{Frank-Wolfe with a Nearest Extreme Point Oracle}
	\maketitle
	
	\begin{abstract}%
		We consider variants of the classical Frank-Wolfe algorithm for constrained smooth convex minimization, that instead of access to the standard  oracle for minimizing a linear function over the feasible set, have access to an oracle that can find an extreme point of the feasible set that is closest in Euclidean distance to a given vector. We first show that for many feasible sets of interest, such an oracle can be implemented with the same complexity as the standard linear optimization oracle. We then show that with such an oracle we can design new Frank-Wolfe variants which enjoy significantly improved complexity bounds in case the set of optimal solutions lies in the convex hull of a subset of extreme points with small diameter (e.g., a low-dimensional face of a polytope). In particular, for many $0\text{--}1$ polytopes, under quadratic growth and strict complementarity conditions, we obtain the first linearly convergent variant with rate that depends only on the dimension of the optimal face and not on the ambient dimension.
	\end{abstract}

	\section{Introduction}
	
	The Frank-Wolfe (FW) algorithm (aka the conditional gradient method) is a classical first-order method for minimzing a smooth and convex function $f(\cdot)$ over a convex and compact feasible set $\mK$ \cite{frank1956algorithm,Polyak,Jaggi13b}, where in this work we assume for simplicity that the underlying space is $\reals^d$ (though our results are applicable to any Euclidean vector space). This algorithm has regained significant interest within the machine learning and optimization communities in recent years due to the fact that, aside of access to a first-order oracle of the objective function, it only requires on each iteration to minimize a linear function over the feasible set which, in many cases of interest, is much more efficient than computing projections, as required by projected/proximal gradient methods. Another benefit of the method is that when the number of iterations is not too high, it produces iterates that are given as an explicit sparse convex combination of extreme points of the feasible set \cite{Jaggi13b}.
	
	The well-known convergence rate of the method is $O(\beta{}D_{\mK}^2/t)$, where $\beta$ is the smoothness parameter of the objective, $D_{\mK}$ is the Euclidean diameter of the set, and $t$ is the iteration counter. It is well-known that this rate is not improvable even if the objective function is strongly convex  (see for instance \cite{lan2013complexity}), a property that is well known to allow for faster convergence rates, and in particular linear rates, for projected/proximal gradient methods \cite{nesterov2018lectures, Beck2017first}. Indeed, in recent years there is a significant research effort to design Frank-Wolfe variants with linear convergence rates under strong convexity or the weaker assumption of quadratic growth (see Definition \ref{def:QG} in the sequel), with most efforts focused on the case in which the feasible set is a convex and compact polytope \cite{GueLat1986,Garber13,Garber16linearly,lacoste2015global,beck2017linearly,GM16, Pena16,Pena19,locally2020,Garber20}.
	
	Despite of the above results, there is a significant disadvantage for both the standard Frank-Wolfe method and its linearly converging variants for polytopes (in comparison to projected/proximal gradient methods) that has not been addressed so far --- the inherent dependency of the convergence rates on the diameter of the feasible set. First, while projected/proximal gradient methods enjoy the benefit of ``warm-start" initialization, i.e., their convergence rates often depend on the distance of the initialization point from the optimal set (see for instance \cite{nesterov2018lectures, Beck2017first}), for the Frank-Wolfe method we do not get such a dependence  and even with a ``warm-start" initialization, the rate depends on the diameter of the entire feasible set\footnote{In \cite{freund2016new} it was shown that with a modified step-size, a ``warm-start'' could be leveraged to reduce the number of iterations required by Frank-Wolfe to reach a desired approximation error by an additive constant, however the resulting rate still depends on the diameter of the entire set.}. Second, existing linearly convergent variants for polytopes depend on the diameter of the polytope, i.e., the convergence rate is of the form $\exp(-\Theta(D_{\mK}^{-2}t))$, where $\DK$ is the diameter (e.g.,  \cite{Garber16linearly,lacoste2015global,GM16}), which is in stark contrast to proximal/projected gradient methods which, under strong convexity/quadratic growth of objective, enjoy a linear convergence rate with exponent that is independent of the diameter. 
	
	Such inferior dependence on the diameter  is of great importance in many setups of interest for Frank-Wolfe-type methods. 
	As a running example, consider polytopes that arise naturally from combinatorial structures such as the flow polytope (convex-hull of source-target paths in a directed acyclic graph), the spanning trees polytope of a graph, the perfect matchings polytope of a bipartite graph, or the base-polyhedron of a matroid. In all of these cases, solving the Frank-Wolfe linear optimization step can be done very efficiently using simple well-known combinatorial algorithms. However, for all of these polytopes the Euclidean diameter is in worst-case $\Theta(\sqrt{n})$, where $n$ is the number of vertices in the above-mentioned graph-induced polytopes, and the size of the bases of the matroid in case of the base-polyhedron. Thus, with high-dimensional problems in mind, it is of clear interest to improve the complexity of Frank-Wolfe-type methods in terms of the diameter.
	
	Our approach towards tackling this challenge is to consider Frank-Wolfe variants with a seemingly stronger oracle than the standard linear optimization oracle. As we shall discuss in the sequel, it turns our that in many setups of interest, this stronger oracle could be implemented with the same complexity as the standard linear optimization oracle.
	
	Concretely, let us denote the set of extreme points of the feasible set by $\mV$. The Frank-Wolfe method assumes the availability of an oracle that solves the following linear optimization problem over $\mV$:
	\begin{align}\label{eq:linearOracle}
		\argmin_{\v\in\mV}\v^{\top}\nabla{}f(\x),
	\end{align} 
	where $\x$ is some feasible point. We emphasize that the linear problem is solved over the set of extreme points (and not entire feasible set), since in most cases of interest indeed an efficient implementation of the linear optimization oracle will only consider the extreme points of the set (e.g., in all combinatorial polytopes mentioned above and for other settings of interest such as the nuclear norm ball of matrices or the set of trace-bounded positive semidefinite matrices \cite{Jaggi13b}).
	
	In this paper we suggest using the following oracle which minimizes an $\ell_2$-regularized version of \eqref{eq:linearOracle}:
	\begin{align}\label{eq:quadOracle}
		\argmin_{\v\in\mV}\v^{\top}\nabla{}f(\x) + \lambda\Vert{\v-\x}\Vert^2,
	\end{align} 
	where $\Vert\cdot\Vert$ denotes the Euclidean norm and $\lambda >0$. Note that \eqref{eq:quadOracle} could be rewritten as
	\begin{align}\label{eq:NEPOracle}
		\argmin_{\v\in\mV}\Vert{\v-\big(\x-\frac{1}{2\lambda}\nabla{}f(\x)\big)}\Vert^2.
	\end{align} 
	That is, we assume that given some point $\y$ we can efficiently find the nearest extreme point (NEP) in the feasible set. We thus refer to an oracle for solving \eqref{eq:NEPOracle} as a NEP oracle. 
	Here it is very important to emphasize that the optimization problems \eqref{eq:quadOracle}, \eqref{eq:NEPOracle} are only solved w.r.t. the set of extreme points --- the set $\mV$. This is very different from the regularized oracles suggested before in \cite{migdalas94, lan2013complexity}  which solve certain regularized problems w.r.t. the entire feasible set, and thus have complexity similar to that of computing projection over the set, which is exactly the computational bottleneck that Frank-Wolfe-type methods aim to avoid. We also note that an oracle of the form \eqref{eq:quadOracle} was implicitly already considered in \cite{garber2016faster}, but there it was for the specific case in which the feasible set is the spectrahedron (unit-trace positive semidefinite matrices) and the aim was to obtain faster rates in terms of the iteration counter $t$.
	
	Let us denote by $\mX^*$ the optimal set, that is $\mX^*=\arg\min_{\x\in\mK}f(\x)$. We also denote $f^*=\min_{\x\in\mK}f(\x)$. Define
	\begin{align}\label{eq:Sstar}
		S^*\in{\argmin}_{S\subseteq\mV} \{\max_{\u,\v\in S}\Vert{\v-\u}\Vert:\mX^*\subseteq\conv(S)\}, \qquad
		D^*=\max_{\u,\v\in S^*} \Vert{\v-\u}\Vert.
	\end{align}
	That is, $S^*$ is a subset of extreme points of minimum diameter whose convex-hull contains the optimal set $\mX^*$, and $D^*$ is the corresponding diameter. Also, when $\mK$ is a convex and compact polytope, we let $\mF^*$ denote the lowest-dimensional face of $\mK$ such that $\mX^*\subseteq\mF^*$ and we let $D_{\mF^*}$ denote the Euclidean diameter of $\mF^*$.
	As we shall discuss in detail in the sequel, in many cases of interest we have that $D^* \ll \DK$ or $D_{\mF^*} \ll \DK$. For instance, when there is a unique optimal solution which is an extreme point we have that $D^*=0$, or very often in case $\mK$ is a polytope and $\dim\mF^* \ll d$ --- a natural notion of sparsity for polytopes.
	
	Our main contributions, in an informal and simplified presentation, are as follows:
	\begin{enumerate}
		\item
		We show that for many feasible sets of interest the NEP oracle \eqref{eq:NEPOracle} could be implemented with the same complexity as the standard optimization linear oracle \eqref{eq:linearOracle}. Cases of interest include $0\text{--}1$ polytopes, nuclear norm balls of matrices, bound-trace positive semidefinite matrices, the unit spectral norm ball of matrices, and unstructured convex hulls. See Section \ref{sec:examples}.
		\item 
		We present a natural variant of the Frank-Wolfe method which replaces the linear optimization oracle with the proposed NEP oracle and enjoys an improved rate of $O(\beta(D^{*2}+D_L^2)/t)$, where $D_L$ is the diameter of the initial level set.\footnote{That is $D_L=\max\{\Vert{\u-\v}\Vert~|~\u\in\mK,\v\in\mK,f(\u)\leq f(\x_1), f(\v)\leq f(\x_1)\}$, where $\x_1$ is the initialization point.\label{footnote1}} Assuming quadratic growth of the objective w.r.t. the feasible set, this rate improves to $O(\beta{}D^{*2}/t)$ plus a lower-order term of the form $\log(t)/t^2$. We also show such rates are not possible to obtain via Frank-Wolfe variants using a linear optimization oracle. See Theorems \ref{thm:alg1 rate} and \ref{thm: theorem 1}.
		\item 
		In case $\mK$ is a polytope and quadratic growth holds, we present a linearly convergent FW variant which replaces the linear optimization oracle with a NEP oracle and enjoys a rate of the form $\exp(-t/(dD_{\mF^*}^{2}))$, which improves upon the previous best $\exp(-t/(dD_{\mK}^2))$ (e.g., \cite{Garber13,Garber16linearly,lacoste2015global}.\footnote{Here for simplicity we hide dependencies on certain geometric quantities of the polytope which are standard for such results, as well as on the objective's condition number $\beta/\alpha$, where $\alpha$ is the quadratic growth parameter.} For many polytopes of interest this gives the first linearly convergent algorithm with rate independent of the diameter of the entire polytope and whose per-iteration oracle complexity matches that of Frank-Wolfe. See Theorem \ref{thm: linrate alg3}.
		
		\item
		In case $\mK$ is a polytope and both quadratic growth and strict complementarity (see definition in the sequel) hold, our linearly convergent variant converges with rate  $\exp(-t/(\dim\mF^*D_{\mF^*}^{2}))$, where $\dim\mF^*$ is the dimension of the optimal face $\mF^*$. As a consequence, for many $0\text{--}1$ polytopes we obtain the first FW variant whose convergence rate is independent of the dimension, provided that $\mF^*$ is low-dimensional. See Theorem \ref{thm: alg3 lin delta}.
		
		\item
		We demonstrate that a NEP oracle could also lead to similar improvements to those in item 2 above in the stochastic setting, by analyzing a variant of the Stochastic Frank-Wolfe method. See Theorem \ref{thm: stoc improved rate}.
	\end{enumerate}
	
	
	
	\subsection{Notation}
	We use boldface lowercase letters to denote vectors and lightface letters to denote scalars. When considering the space of matrices $\reals^{m\times n}$ or that of symmetric $n\times n$ matrices $\mbS^n$, we use  boldface uppercase letters to denote matrices. Throughout, we let $\Vert{\cdot}\Vert$ denote the Euclidean norm. For matrices we let $\Vert{\cdot}\Vert_2$ denote the spectral norm (i.e., largest singular value) and $\Vert{\cdot}\Vert_F$ denote the (Euclidean) Frobenius norm. For a set $S$ of points in a Euclidean vector space we let $\conv(S)$ denote their convex-hull. Given a point $\x$ and convex and compact set $\mS$ in a Euclidean vector space, we let $\dist(\x,\mS)$ denote the Euclidean distance of $\x$ from $\mS$.

	\subsection{Examples of feasible sets of interest}\label{sec:examples}
	
	\paragraph*{0\text{--}1 polytopes:} $0\text{--}1$ polytopes are polytopes which satisfy $\mV\subseteq\{0,1\}^d$, i.e., all vertices have either 0 or 1 entries. This family of polytopes captures many combinatorial polytopes of interest including the flow polytope, the spanning-tree polytope, the perfect matchings polytope of a bipartite graph, the base polyhedron of a matroid, the unit simplex, the hypercube $[0,1]^d$, and many more.
	
	For such polytopes, the NEP oracle could be implemented directly using a linear optimization oracle since for any vector $\y$ we have
	\begin{align*}
		\argmin_{\v\in\mV}\Vert{\v-\y}\Vert^2 = \argmin_{\v\in\mV}\Vert{\v}\Vert^2 - 2\v^{\top}\y = \argmin_{\v\in\mV}\v^{\top}\mathbf{1} - 2\v^{\top}\y = \argmin_{\v\in\mV}\v^{\top}(\mathbf{1}-2\y),
	\end{align*}
	where $\mathbf{1}$ denotes the all-ones vector.
	
	For many of these polytopes (e.g., flow polytope, perfect matchings polytope of a bipartite graph, hypercube, but not the simplex), the diameter of the optimal face scales with its dimension, e.g.,  $\textrm{diam}(\mF^*) = \Theta(\sqrt{\dim\mF^*})$. Thus, when the optimal set $\mX^*$ is contained within a low-dimensional face $\mF^*$, we indeed have that $D^*\leq D_{\mF^*} \ll D_{\mK}$. Low-dimensionality of the optimal face can be seen as a natural notion of sparsity (since it implies any optimal solution can be expressed as a sparse convex combination of extreme points), which is an important concept in many machine learning / statistical settings (see  also discussions in the recent work \cite{Garber20}).
	
	\paragraph*{The spectrahderon and trace-norm balls:} The spectrahderon $\{\X\in\mbS^n~:~\X\succeq 0,~\trace(\X)=\tau\}=\conv\{\tau\u\u^{\top}~|~\u\in\reals^n,~\Vert{\u}\Vert=1\} $ and the nuclear norm ball $\{\X\in\reals^{m\times n}~|~\Vert{\X}\Vert_*\leq \tau\}=\conv\{\tau\u\v^{\top}~|~\u\in\reals^m,\v\in\reals^n,~\Vert{\u}\Vert=\Vert{\v}\Vert=1\}$, where $\tau >0$, are two highly popular convex relaxations for matrix rank constraint, and are ubiquitous in convex relaxations for low-rank matrix recovery problems (e.g., low-rank matrix completion). Optimization over these sets is one of the main reasons for the increasing popularity of Frank-Wolfe-type methods, since linear optimization over these sets amounts to a rank-one SVD computation --- a task for which there exists very efficient iterative methods (e.g., power iterations, Lanczos algorithm), while Euclidean projection requires in general a full-rank SVD computation \cite{jaggi2010simple,garber2016faster,allen2017linear}.
	Since all extreme points of these sets have the same Euclidean norm, the NEP oracle is equivalent to the linear optimization oracle --- see \textit{Norm-uniform sets} in the sequel, and hence could be implemented with the same complexity.
	
	Consider now the case in which the feasible set is the spectrahedron and suppose that the optimal solution is unique and can be written in the form $\X^*=\sum_{i=1}^m\lambda_i\tau\u_i\u_i^{\top}$, where $(\lambda_1,\dots,\lambda_m)$ is in the unit simplex, $\Vert{\u_i}\Vert=1, i=1,\dots,m$, and for all $i\neq j$ we have $(\u_i^{\top}\u_j)^2 \geq 1-\gamma$, for some $\gamma > 0$. In this case, it follows that $D^* \leq \max_{i,j}\Vert{\tau\u_i\u_i^{\top}-\tau\u_j\u_j^{\top}}\Vert_F \leq \sqrt{2\gamma}\tau$. Thus, for $\gamma \ll 1$ (which means $\X^*$ admits a crude approximation via a rank-one matrix), we can indeed have $D^* \ll \sqrt{2}\tau = D_{\mK}$. Clearly, a similar argument holds for the nuclear norm ball as well.

	\paragraph*{Unit spectral norm ball:} Consider the set of matrices $\mK = \{\X\in\reals^{m\times n}~|~\Vert{\X}\Vert_2\leq 1\} = \conv\{\U\in\reals^{m\times n}~|~\U^{\top}\U=\I\}$, $m\geq n$. Linear optimization over this set amounts to a SVD computation \cite{Jaggi13b}. Since, as in the previous example, all extreme points have the same Euclidean norm, the NEP oracle is equivalent to the linear optimization oracle (see \textit{Norm-uniform sets} next). While for this set we have $D_{\mK} = 2\sqrt{n}$, similarly to the example for the spectrahedron above, we can clearly have $D^* \ll 2\sqrt{n}$. 
	
	\paragraph*{Norm-uniform sets:} In case all extreme points have the same Euclidean norm i.e., $\Vert{\u}\Vert=\Vert{\v}\Vert$ for all $\u,\v\in\mV$, we clearly have that the NEP oracle could be implemented using a single call to the standard linear optimization oracle since for any vector $\y$:
	\begin{align*}
		\argmin_{\v\in\mV}\Vert{\v-\y}\Vert^2 = \argmin_{\v\in\mV}\Vert{\v}\Vert^2 - 2\v^{\top}\y = \argmin_{\v\in\mV}\v^{\top}(-2\y).
	\end{align*}
	Note that many of the above examples (e.g., perfect matchings polytope, base polyhedron of a matroid, spectrahedron, nuclear norm ball, unit spectral norm ball etc.) satisfy this property. Clearly, $\ell_1$, $\ell_2$ and $\ell_{\infty}$ balls in $\reals^d$ also satisfy this property.
	
	\paragraph*{Unstructured convex-hulls:} In case the feasible set $\mK$ is  given by its set of extreme points without further structure, i.e., $\mK=\conv\{\mV\}$, where the set $\mV$ is explicitly given, linear optimization amounts to computing the linear function over all extreme points and taking the minimum. Clearly, we can also implement the NEP oracle with the same complexity.

	\subsection{The quadratic growth property}
	Most  (but not all) results we report in this paper hold under the quadratic growth property which has similar consequences to strong convexity but is a weaker assumption.
	\begin{definition}[quadratic growth]\label{def:QG} 
		We say a function $f:\mK\rightarrow\reals$ satisfies the quadratic growth property with parameter $\alpha>0$ w.r.t. a convex and compact set $\mK$ if for all $\x\in\mK$ it holds that $ \dist(\x,\mX^*)^2 \leq 2\alpha^{-1}\left({f(\x)-f^*}\right)$.
	\end{definition}
	
	Quadratic growth is known to hold whenever $f(\x)$ is of the form $f(\x)=g(\A\x)+\b^{\top}\x$ for $\alpha_g$-strongly convex $g:\reals^m\rightarrow\reals$, $\A\in\reals^{m\times d}$, and $\mK$ is a convex and compact polytope (see for instance \cite{beck2017linearly,Garber19}). Very recently, it was also established that for the important matrix domains --- the spectrahedron and nuclear norm ball (mentioned above), for $f(\cdot)$ of this form which also satisfies a certain strict complementarity condition, quadratic growth also holds, see \cite{ding2020spectral,ding2020k,Garber2019linear}. Note that such structure of $f(\cdot)$ holds in particular for $g(\z) =\Vert{\z-\y}\Vert^2$ and $\A$ being an underdetermined linear map which captures some of the most fundamental problems in statistics and machine learning.
	
	\section{Main Results}
	In this section we present our novel algorithms and corresponding convergence rates. The complete proofs, as well as additional results, are given in the appendix.
	
	\subsection{A NEP Oracle-based Frank-Wolfe Variant for General Convex Sets}
	Our first line of results concerns a straightforward adaptation of the standard Frank-Wolfe method, where the call to the linear optimization oracle is replaced with a call to the NEP oracle, see Algorithm \ref{alg:Algorithm 1} below. We note that we use a modified step-size $\tilde{\eta}_t$  for the convex combination on each iteration (and not the one used to compute the new extreme point $\v_t$) in order to guarantee that out method is a decent method which is important for achieving the improved dependencies on the initialization point. 
	
	\begin{algorithm}[H]
		\caption{Frank-Wolfe with Nearest Extreme Point Oracle}
		\label{alg:Algorithm 1}
		\begin{algorithmic}[1]
			\STATE Input: sequence of step-sizes $\{\eta_t\}_{t\geq 1} \subset [0,1]$
			\STATE $\x_1 \leftarrow$ some arbitrary point in $\mV$
			\FOR {$t=1,2, \dotsc $}
			\STATE $\v_t \leftarrow \argmin_{\v\in \mV} \v^\top\nabla f(\x_t) + \frac{\beta\eta_t}{2}\lVert \x_t - \v \rVert^2$ \\
			\COMMENT{equivalent to $\v_t\gets\argmin_{\v\in\mV}\Vert{\v-(\x_t - (\beta\eta_t)^{-1}\nabla{}f(\x_t))}\Vert^2$}
			\STATE pick $\tilde{\eta}_t\in[0,1]$ such that $f((1-\tilde{\eta}_t)\x_t + \tilde{\eta}_t \v_t) \leq \min\{f((1-{\eta}_t)\x_t + {\eta}_t \v_t), f(\x_t)\}$
			\STATE $\x_{t+1} \leftarrow (1-\teta_t)\x_t + \teta_t \v_t$
			\ENDFOR
		\end{algorithmic}
	\end{algorithm}
	
	\begin{theorem}\label{thm:alg1 rate}
		Using Algorithm~\ref{alg:Algorithm 1} with step-size $\eta_t=\frac{2}{t+1}$ we have
		\begin{align*}
			\forall t\geq 2: \quad f(\x_t)-f^* \leq \frac{2\beta(\DO^2+D_L^2)}{t+1},
		\end{align*}
		where $D_L$ is the diameter of the initial level set (see Footnote \ref{footnote1}).
		
		Moreover, if $f(\cdot)$ has the quadratic growth property over $\mK$ with parameter $\alpha >0$, then 
		\begin{align*}
			\forall t\geq 2:\quad  f(\x_t)-f^*\leq \frac{2\beta\DO^2}{t+1} + \frac{\frac{8\beta^2}{\alpha}(\DO^2+\min\{2\alpha^{-1}(f(\x_1)-f^*),D_L^2\})\log(t)}{t^2}.
		\end{align*}
	\end{theorem}
	Theorem \ref{thm:alg1 rate} shows that as opposed to the standard convergence rate of the Frank-Wolfe method which, regardless of the initialization, is $O(\beta{}D_{\mK}^2/t)$, our NEP oracle-based variant has a rate of the form $O(\beta{}D^{*2}/t)$ plus an additional natural term that depends on the quality of the initialization point. In particular, under quadratic growth, this additional term decays at a fast rate of $O(\log{}t/t^2)$, and thus even without good initialization, our algorithm has significantly improved complexity whenever $D^* \ll D_{\mK}$.
	
	\subsubsection{Complementary lower bound for linear optimization-based FW variants}
	We now present a complementary result showing that the rates reported in Theorem \ref{thm:alg1 rate} are impossible to obtain in general for Frank-Wolfe variants using only a linear optimization oracle.
	
	\begin{definition}[Frank-Wolfe-type method (see also \cite{Garber20})]\label{def:fwmethod}
		An iterative algorithm for the optimization problem $\min_{\x\in\mK}f(\x)$, where $\mK$ is convex and compact and $f(\cdot)$ is smooth and convex, is a Frank-Wolfe-type method if on each iteration $t$, it performs a single call to the linear optimization oracle of $\mK$ w.r.t. the point $\nabla{}f(\x_t)$, i.e., computes some $\v_t\in\argmin_{\v\in\mK}\v^{\top}\nabla{}f(\x_t)$, where $\x_t$ is the current iterate, and produces the next iterate $\x_{t+1}$ by taking some convex combination of the points in $\{\x_1,\v_1,\dots,\v_t\}$, where $\x_1$ is the initialization point.
	\end{definition}

	\begin{theorem}\label{thm: theorem 1}
		Let $\mK=[0,1]^d$, and fix a positive integer $m<d$. Let $\xo=\frac{1}{2}\sum_{i=1}^{m} \e_i$, i.e., $\xo$ has $\frac{1}{2}$ for the first $m$ coordinates and 0 for the rest.
		Now consider the minimization of the function $f(\x) = \frac{1}{2}\lVert \x - \xo \rVert^2$ over $\mK$ starting at $\x_1=\e_{m+1}$. Then, for any Frank-Wolfe-type method there exists a sequence of answers returned by the linear optimization oracle such that for any $t\leq \lfloor \sqrt{d-m-1} \rfloor$, the $t$th iterate of the algorithm $\x_t$ satisfies $f(\x_t)-f^* \geq \frac{1}{4}$. 
	\end{theorem}
	
	Note that for the problem described in Theorem \ref{thm: theorem 1} it holds that $f(\x_1)-f^* = \frac{1}{2}+\frac{m}{8}$, $D^*=\sqrt{m}$, and $\alpha=\beta=1$. Thus, when $m\ll \sqrt{d}$, we have that $O(m)$ iterations suffice for Algorithm \ref{alg:Algorithm 1} to obtain approximation error $< 1/4$, while any Frank-Wolfe-type method with a linear optimization oracle will require $\Omega(\sqrt{d})$ iterations. Moreover, all currently existing upper-bounds for Frank-Wolfe-type methods actually require $\Omega(d)$ iterations (since $D_{\mK}^2 = d$ in this case).
	
	\subsection{A NEP Oracle-based Linearly-Convergent Variant for Polytopes}\label{sec:linres}
	We now turn to discuss our NEP oracle-based linearly convergent algorithm for the case in which the feasible set $\mK$ is a convex and compact polytope.
	Our algorithm (see Algorithm \ref{alg:Algorithm 3}) is an adaptation of the \textit{fully-corrective} Frank-Wolfe variant \cite{Jaggi13b}, where the linear optimization step is replaced in a straightforward manner with the new NEP oracle (similarly to Algorithm \ref{alg:Algorithm 1}). On each iteration of the algorithm we optimize either the quadratic-upper bound on the objective due to smoothness (Option 1) or the objective itself (Option 2) over the convex-hull of all previously found vertices. 
	
	\begin{algorithm}
		\caption{Linearly Convergent Frank-Wolfe for Polytopes with a NEP Oracle}
		\label{alg:Algorithm 3}
		\begin{algorithmic}[1]
			\STATE Input: a sequence $\{\rho_t\}_{t\geq 1} \subset[0,1]$
			\STATE $\x_1 \leftarrow$ some arbitrary point in $\mV$
			\FOR {$t=1,2, \dotsc $}
			\STATE Let $\sum_{i=1}^k \lambda_i\v_i$ be an explicitly maintained decomposition of $\x_t$ to vertices
			\STATE $\v_{k+1}\leftarrow {\argmin}_{\u\in\mV} \u^\top\nabla f(\x_t)+\beta\rho_t\Vert{\u-\x_t}\Vert^2$
			\STATE \textbf{Option 1:} 
			\STATE \hspace{\algorithmicindent}$\x_{t+1}\leftarrow{\argmin}_{\x\in\conv(\v_1,\dots,\v_{k+1})} \x^\top\nabla f(\x_t)+\frac{\beta}{2}\Vert{\x-\x_t}\Vert^2$
			\STATE \textbf{Option 2 (fully-corrective):}
			\STATE\hspace{\algorithmicindent} $\x_{t+1}\leftarrow{\argmin}_{\x\in\conv(\v_1,\dots,\v_{k+1})} f(\x)$
			\ENDFOR
		\end{algorithmic}
	\end{algorithm}
	
	We make some comments regarding the efficient solution of the quadratic optimization problem in lines 7 and 9 of Algorithm \ref{alg:Algorithm 3}. Regarding line 7,
	when the number of vertices in the decomposition of the current iterate $k$ satisfies $k \ll d$, using a preprocessing step (i.e., explicitly computing $\v_i^{\top}\nabla{}f(\x_t)$, $\v_i^{\top}\x_t$, $\v_i^{\top}\v_j$, $i=1,\dots,k+1,j=1,\dots,k+1$), this problem could be reformulated as convex quadratic minimization over the $k$-dimensional unit simplex. Such problems can be efficiently solved via fast first-order methods. The problem in line 9 can also be reformulated as  a convex problem over the $k$-dimensional unit simplex however,  solving it via first-order methods requires more evaluations of the objective's gradients.
	
	In order to present our guarantees for Algorithm \ref{alg:Algorithm 3}, we require some additional notation. Following \cite{Garber16linearly,Garber20} we assume $\mK$ is a convex and compact polytope in the form $\mK := \{\x\in\reals^d ~|~\A_1\x = \b_1,~\A_2\x\leq \b_2\}$, $\A_1\in\reals^{m_1\times{}d}$, $\A_2\in\reals^{m_2\times{}d}$. For a face $\mF$ of $\mK$ we define:
	\begin{align*}
		\dim\mF := d - \dim\textrm{span}\{&\{\A_1(1),\cdots\A_1(m_1)\}\cup\\
		&\{\A_2(i): i\in[m_2], \forall\x\in\mF:\A_2(i)^{\top}\x=\b_2(i)\}\}.
	\end{align*}
	We let $\mF^*\subseteq\mK$ denote the lowest-dimensional face of $\mP$ containing the set of optimal solutions, i.e., $\mX^*\subseteq\mF^*$.
	In the following we write $\mF^* = \{\x\in\reals^d~|~\A_1^*\x=\b_1^*,~\A_2^*\x\leq\b_2^*\}$.\footnote{The rows of $\A_1^*$ are exactly the rows of $\A_1$ plus rows of $\A_2$ which correspond to constraints that are tight for all points in $\mF^*$ and the vector $\b_1^*$ is defined accordingly.  The rows of $\A_2^*$ are the rows of $\A_2$ which correspond to constraints that are satisfied by some of the points in $\mF^*$ but not by others, and the vector $\b_2^*$ is defined accordingly.} 
	We let $\mathbb{A}^*$ denote the set of all $\dim\mF^*\times d$ matrices whose rows are linearly independent rows chosen from the rows of $\A_2^*$.  We define $\psi = \max_{\M\in\mathbb{A}^*}\Vert{\M}\Vert_2$ and 
	$\xi = \min_{\v\in\mV\cap\mF^*}\min_i\{{\b_2^*(i) - \A_2^*(i)}^{\top}\v ~|~ \b_2^*(i) > {\A_2^*(i)}^{\top}\v\}$. Finally, we let $D_{\mF^*}$ denote the diameter of the optimal face.
	
	\begin{theorem}\label{thm: linrate alg3}
		Suppose that $\mK$ is a convex and compact polytope and quadratic growth holds with parameter $\alpha > 0$. Let $C\geq f(\x_1)-f^*$ and $M\geq\max\{\frac{\beta}{\alpha}(4+8d\mu^2\DFO^2),\frac{1}{2}\}$, where $\mu=\frac{\psi}{\xi}$.\footnote{Note that for polytopes such as the flow polytope, the perfect matchings polytope of a bipartite graph, the $[0,1]^d$ hypercube and the unit simplex it holds that $\mu=1$.} Using Algorithm~\ref{alg:Algorithm 3} with parameter $\rho_t=\frac{\min\{\sqrt{\frac{2Cd\mu^2}{\alpha}\exp\big({-\frac{1}{4M}(t-1)}\big)},1\}}{2M}$  for all $t\geq 1$
		one has,
		\begin{align*}
			\forall t\geq1: \quad f(\x_t)-f^*\leq C\exp\Big(-\frac{t-1}{4M}\Big).
		\end{align*} 
	\end{theorem}
	
	Theorem \ref{thm: linrate alg3} improves upon the state-of-the art complexity bounds for Frank-Wolfe-type methods for general polytopes \cite{Garber16linearly}\footnote{We note that i. while other results on linearly converging FW variants have complexity bounds that are stated using different quantities, such as the \textit{pyramidal width} in \cite{lacoste2015global}, their worst-case complexity bounds do not improve over \cite{Garber16linearly}, and ii. while \cite{GM16} presented a FW variant with improved dependence on the dimension (but without improvement in dependene on $D_{\mK}$), their result applies only to a very restricted family of polytopes and in particular a strict subset of the $0\text{--}1$ polytopes.}   by replacing the dependency on $D_{\mK}^2$ with $D_{\mF^*}^{*2}$. For instance, for the $[0,1]^d$ hypercube, when $\dim\mF^* \ll d$ we  have $D_{\mF^*}^{2}=\Theta(\dim\mF^*) \ll D_{\mK}^2 = \Theta(d)$.
	
	We note  that while Theorem \ref{thm: linrate alg3} relies on a pre-defined sequence of step-sizes $\{\rho_t\}_{t\geq 1}$ which can be difficult to tune in practice, in Theorem \ref{thm:alg3 param} (Section \ref{sec:adaptiveLinear}) we prove that the same rate can be obtained  using an adaptive step-size by applying a logarithmic-scale search on each iteration $t$ to choose a value for $\rho_t$ which gives the largest decrease in function value.
	

	\subsubsection{Improved dependence on dimension under strict complementarity}	
	While for many polytopes  Theorem \ref{thm: linrate alg3} implies significant improvement in the dependence on the dimension whenever $\dim\mF^* \ll d$, still the exponent has explicit dependence on the dimension $d$. 
	
	In a very recent work \cite{Garber20} it was shown that even when the optimal face is low-dimensional, without further assumptions, Frank-Wolfe-type methods (as defined in Definition \ref{def:fwmethod}) cannot avoid such dependence. It was also shown that under a strict complementarity condition (see Assumption \ref{asm: dstrict}), it is possible to improve the explicit dependence on the dimension $d$ to only dependence on the dimension of the optimal face $\mF^*$. The strict complementarity assumption, in the context of analyzing Frank-Wolfe-type methods, was suggested by Wolfe himself  \cite{Wolfe1970}, and it was also instrumental in the early work \cite{GueLat1986} on linearly-converging Frank-Wolfe methods, but not in the more modern ones such as \cite{Garber13, Garber16linearly,lacoste2015global}. \cite{Garber20} motivated this assumption by proving it implies the robustness of the optimal face $\mF^*$ to small perturbations in the objective function $f(\cdot)$. 
	
	\begin{assumption}[strict complementarity]\label{asm: dstrict}
		There exist $\delta>0$ such that for all $\xo\in\mX^*$ and $\v\in\mV$: if $\v\in\mV\setminus\FO$ then $(\v-\xo)^\top\nabla f(\xo)\geq\delta$; otherwise, if $\v\in\mV\cap\FO$ then $(\v-\xo)^\top\nabla f(\xo)=0$.
	\end{assumption}	
	
	\begin{theorem}\label{thm: alg3 lin delta}
		Suppose that in addition to the assumptions of Theorem \ref{thm: linrate alg3}, Assumption~\ref{asm: dstrict} also holds with some parameter $\delta>0$, and let $M_1\geq\max\{\frac{4\beta}{\alpha}+8\beta\DFO^2\max\{2\kappa,\delta^{-1}\},\frac{1}{2}\}$, $M_2\geq\{\frac{4\beta}{\alpha}+16\beta\kappa\DFO^2,\frac{1}{2}\}$, and some $C\geq f(\x_1)-f^*$,  where $\kappa=\frac{2\mu^2\dim\FO}{\alpha}$.\\
		Using Algorithm~\ref{alg:Algorithm 3} with parameters $\rho_t=\frac{\min\{\sqrt{2\max\{2\kappa,\delta^{-1}\}C\exp\big(-\frac{1}{4M_1}(t-1)\big)},1\}}{2M_1}$ for all $t\geq 1$, one has,
		\begin{align}\label{eq: alg3 lin delta res 1}
			\forall t\geq 1:\quad f(\x_t)-f^*\leq C\exp\Big(-\frac{t-1}{4M_1}\Big).
		\end{align}
		Furthermore, in case that $2\kappa\leq\delta^{-1}$ and $\dim\FO>0$ (i.e. $\kappa>0$), denoting $\tau\geq 4M_1\log(\frac{C}{\delta^2\kappa})+1$ and for all $t\geq \tau$, using $\rho_t=\frac{\min\{2\delta\kappa \exp(-\frac{t-\tau}{8M_2}),1\}}{2M_2}$ instead of the above value, one has,
		\begin{align}\label{eq: alg3 lin delta res 2}
			\forall t\geq \tau:\quad f(\x_t)-f^*\leq\delta^2\kappa \exp\Big(-\frac{t-\tau}{4M_2}\Big).
		\end{align}
	\end{theorem}
	
	Theorem \ref{thm: alg3 lin delta} improves upon \cite{Garber20} by replacing the dependence on $D_{\mK}$ with $D_{\FO}$. Thus, for many polytopes of interest, when $\dim\mF^* \ll d$, Theorem \ref{thm: alg3 lin delta} has no explicit dependence on the ambient dimension $d$, but only on the dimension of the low-dimensional optimal face $\mF^*$. In particular, as discussed, for many polytopes of interest, the NEP oracle could be implemented using a single call to the linear optimization oracle, and as a result, for many $0\text{--}1$ polytopes, under quadratic growth and strict complementarity, we obtain the first linearly convergent algorithm with the same per-iteration oracle-complexity as Frank-Wolfe and with dimension-independent convergence rate whenever $\dim\mF^* \ll d$.
	
	\subsection{A NEP Oracle-based Frank-Wolfe Variant for Stochastic Optimization}\label{sec:stoc results}
	Our last result concerns a standard stochastic optimization setting in which $f(\cdot)$ is given by a stochastic first-order oracle with stochastic gradients upper-bounded in $\ell_2$ norm by some $G>0$. 
	
	Stochastic Frank-Wolfe-type methods have also received notable attention in recent years, including the use of the conditional-gradient sliding approach and various variance reduction methods, see for instance \cite{lan2016conditional,hazan2016variance}. Here however, in order to demonstrate the benefit of our NEP oracle to this setting as well, we only focus on the most basic method known as the Stochastic Frank-Wolfe (SFW) algorithm, which replaces the exact gradient in the Frank-Wolfe method with a mini-batched stochastic gradient (see \cite{hazan2016variance}). Also, here for simplicity we only focus on the case in which $f(\cdot),\mK$ satisfy the quadratic growth property.
	
	Our algorithm, which is a straight-forward adaptation of the SFW algorithm is given as Algorithm~\ref{alg:Algorithm stochastic} (replacing the call to the linear optimization oracle with a call to the NEP oracle) in Appendix \ref{sec:stoc results}.
	
	\begin{theorem}\label{thm:stoc eps convergence}
		Suppose $f,\mK$ satisfy the quadratic growth property with some $\alpha >0$. Using Algorithm~\ref{alg:Algorithm stochastic} with step-size $\eta_t=\frac{2}{t+1}$ and mini-batch sizes that satisfy 
		\begin{align*}
			m_t \geq \max\Big\{\Big(\frac{G(t+1)}{\beta\DK}\Big)^2,\min\bigg\{\Big(\frac{G\DK(t+1)}{\beta\DO^2}\Big)^2,\Big(\frac{\alpha G(t+1)^2}{8\beta^2 \DK}\Big)^2\Big\}\Big\},
		\end{align*} 
		for any $0<\epsilon\leq \frac{16\beta^2\DK^2}{\alpha}$,  expected approximation error $\epsilon$ is achieved after
		$\tilO\big(\max\{\frac{\beta\DO^2}{\epsilon}, \frac{\beta\DK}{\sqrt{\alpha\epsilon}}\}\big)$
		calls to the NEP oracle and
		$\tilO\big(\beta G^2\max\big\{\frac{\DK^2\DO^2}{\epsilon^3}, \frac{\DK}{\alpha^{3/2}\epsilon^{3/2}},\frac{\DK^3}{\alpha^{1/2}\epsilon^{5/2}}\big\}\big)$
		stochastic gradient evaluations, where $\tilO$ suppresses poly-logarithmic terms in $\frac{\beta^2\DK^2}{\alpha\epsilon}$. 
	\end{theorem}
	
	Let us compare Theorem \ref{thm:stoc eps convergence} with the standard SFW method which requires $O(\beta{}G^2D_{\mK}^4/\epsilon^3)$ stochastic gradients and $O(\beta{}D_{\mK}^2/\epsilon)$ calls to the linear optimization oracle \cite{hazan2016variance}, and the state-of-the-art --- the stochastic conditional gradient sliding (SCGS) method \cite{lan2016conditional} which requires optimal $O(G^2/(\alpha\epsilon))$ stochastic gradients and $O(\beta{}D_{\mK}^2/\epsilon)$ calls to the linear optimization oracle.  The improvement over SFW is quite clear , both in terms of the stochastic oracle and the optimization oracle (at least when $\alpha$ is not trivially small). While SCGS has clear advantage in terms of the stochastic oracle complexity\footnote{This is not surprising since, as opposed to SCGS, SFW cannot leverage the quadratic growth to improve the stochastic gradient complexity.}, we see that when the main concern is the optimization oracle complexity and $D^* \ll D_{\mK}$, already the simple stochastic scheme in Algorithm \ref{alg:Algorithm stochastic} can have significant advantage over SCGS.  
	
	\section{Proof Ideas}
	In this section we describe the main novel components in the analysis of Algorithms \ref{alg:Algorithm 1},\ref{alg:Algorithm 3} ---  novel bounds on the per-iteration error reduction which are independent of the diameter of the set.
	
	\subsection{Error reduction for Algorithm \ref{alg:Algorithm 1}}
	\begin{lemma}\label{lem: alg1 one step improve}
		Using Algorithm~\ref{alg:Algorithm 1} with any choice of step-sizes $\{\eta_t\}_{t=1}^{\infty}$ one has,
		\begin{align*}
			\forall t\geq 1: \quad f(\x_{t+1})-f^*\leq (1-\eta_t)(f(\x_t)-f^*) + \frac{\beta\eta_t^2}{2}(\dist(\x_t, \mX^*)^2+\DO^2).
		\end{align*}
	\end{lemma}
	\begin{proof}
		Fix some iteration $t\geq 1$ and let $\xo$ be the optimal solution closest to $\x_t$. From the $\beta$-smoothness of $f(\cdot)$, the optimality of $\v_t$, and the definition of the set $S^*$ in \eqref{eq:Sstar}, we have that,
		\begin{align*}
			f(\x_{t+1})  &\leq \min_{\v\in{}S^*}f(\x_t) + \eta_t(\v-\x_t)^{\top}\nabla{}f(\x_t) + \frac{\beta\eta_t^2}{2}\Vert{\v  + \x^* - \x^* - \x_t}\Vert^2 \\
			&=\min_{\v\in{}S^*}f(\x_t) + \eta_t(\v-\x_t)^{\top}\nabla{}f(\x_t) + \frac{\beta\eta_t^2}{2}\big(\Vert{\x_t-\x^*}\Vert^2 \\ 
			&+ 2(\x_t-\x^*)^{\top}(\x^*-\v) + \Vert{\v-\x^*}\Vert^2\big) \\
			&\leq \min_{\v\in{}S^*}f(\x_t) + \eta_t(\v-\x_t)^{\top}\nabla{}f(\x_t) \\
			&+\frac{\beta\eta_t^2}{2}\big(\dist(\x_t,\mX^*)^2 + 2(\x_t-\x^*)^{\top}(\x^*-\v) + D^{*2}\big)\\
			&\underset{(a)}{\leq} f(\x_t) + \eta_t(\x^*-\x_t)^{\top}\nabla{}f(\x_t) \\
			&+\frac{\beta\eta_t^2}{2}(\dist(\x_t,\mX^*)^2 + 2(\x_t-\x^*)^{\top}(\x^*-\x^*) + D^{*2}) \\
			&\leq f(\x_t) - \eta_t(f(\x_t)-f^*) + \frac{\beta\eta_t^2}{2}(\dist(\x_t,\mX^*)^2 + D^{*2}),
		\end{align*}
		where (a) holds since $\xo\in\conv(S^*)$ and since $\v^\top\nabla f(\x_t)+\beta\eta_t(\x_t-\xo)^\top(\xo-\v)$ is linear in $\v$.
		
		Subtracting $f^*$ from both sides concludes the proof.
	\end{proof}
	
	\subsection{Error reduction for Algorithm \ref{alg:Algorithm 3}}
	We now proceed to the main lemma in the proof of Theorems \ref{thm: linrate alg3}, \ref{thm: alg3 lin delta} which will allow us to bound the improvement on each iteration of Algorithm~\ref{alg:Algorithm 3}. 
	\begin{lemma}\label{lem: one step improve alg3}
		Fix some iteration $t\geq 1$ of Algorithm~\ref{alg:Algorithm 3} (when either Option 1 or 2 are used) and consider a convex decomposition of $\x_t$ into vertices: $\x_t=\sum_{i=1}^k\lambda_i\v_i$, and fix some $\eta\in[0,1]$. Suppose there exists $R\leq 1$ such that some $\x^*\in\mX^*$ can be written as $\xo=\sum_{i=1}^k(\lambda_i-\Delta^*_i)\v_i + \sum_{i=1}^k\Delta^*_i\z$ for $\Delta^*_i\in[0,\lambda_i], \z\in\FO$ and $\sum_{i=1}^k\Delta_i^* \leq R$. Then, using $\rho_t=\eta R$ one has that,
		\begin{align*}
			f(\x_{t+1})-f^* \leq (1-\eta)(f(\x_t)-f^*) + \eta^2\beta\left({2\Vert{\x_t-\x^*}\Vert^2 + 4R^2\DFO^2}\right).
		\end{align*}
	\end{lemma}
	This lemma indeed allows us to obtain linear rates since, as it was shown in \cite{Garber16linearly,Garber20}, informally speaking, we can take $R=O(\Vert{\x_t-\xo}\Vert)= O(\sqrt{f(\x_t)-f^*})$.
	\begin{proof}
		First, note that we can assume w.l.o.g. that $\sum_{i=1}^k\Delta_i^* = R$ (see Observation~\ref{obs:linlem help} in Appendix~\ref{sec:linrate alg3}).\\ 
		Now let $\p_t=\xo-R\z+R\v_{k+1}$ and $\y_{t+1}=(1-\eta)\x_t+\eta\p_t$. Note that since $\p_t=\sum_{i=1}^k(\lambda_i-\Delta^*_i)\v_i + \sum_{i=1}^k\Delta^*_i\v_{k+1}$, we have that $\p_t\in\conv(\v_1,\dots,\v_{k+1})$ and thus $\y_{t+1}\in\conv(\v_1,\dots,\v_{k+1})$.\\
		Thus by the $\beta$-smoothness of $f(\cdot)$ and the optimality of $\x_{t+1}$ (as defined in either line 7 or 9 of Algorithm~\ref{alg:Algorithm 3}), we have that, 
		\begin{align}\label{eq:1 linlem}
			f(\x_{t+1}) &\leq
			f(\x_t) + (\y_{t+1}-\x_t)^{\top}\nabla{}f(\x_t) + \frac{\beta}{2}\Vert{\y_{t+1}-\x_t}\Vert^2 \nonumber \\
			&= f(\x_t) + \eta(\xo-R\z+R\v_{k+1}-\x_t)^{\top}\nabla{}f(\x_t) + \frac{\eta^2\beta}{2}\Vert{\xo-R\z+R\v_{k+1}-\x_t}\Vert^2 \nonumber\\
			&\leq f(\x_t) + \eta(\xo-\x_t)^{\top}\nabla{}f(\x_t)+\eta R(\v_{k+1}-\z)^{\top}\nabla{}f(\x_t)\nonumber\\
			&+ \eta^2\beta(\Vert{\xo-R\z-(1-R)\x_t}\Vert^2 + R^2\Vert{\v_{k+1}-\x_t}\Vert^2) \nonumber\\
			&\leq f(\x_t) + \eta(\xo-\x_t)^{\top}\nabla{}f(\x_t) \nonumber\\
			&+ \eta^2\beta(\Vert{\xo-R\z-(1-R)\x_t}\Vert^2 + R^2\Vert{\w^*-\x_t}\Vert^2)
		\end{align}
		where in the last inequality we let $\w^*$ be a vertex in $\mF^*$ such that $\w^{*\top}\nabla{}f(\x_t) \leq \z^\top\nabla{}f(\x_t)$ and we use the fact that since $\rho_t=\eta R$, we have that $\v_{k+1}\in {\argmin}_{\u\in\mV} \u^\top\nabla f(\x_t)+\beta\eta R\Vert{\u-\x_t}\Vert^2$.\\
		Note that 
		\begin{align}\label{eq:2 linlem}
			\Vert{\w^*-\x_t}\Vert^2 \leq 2\Vert{\w^*-\x^*}\Vert^2 + 2\Vert{\x_t-\x^*}\Vert^2 \leq 2\Vert{\x_t-\x^*}\Vert^2 + 2\DFO^2.
		\end{align}
		Also,
		\begin{align}\label{eq:3 linlem}
			&\Vert{\xo-R\z-(1-R)\x_t}\Vert^2 = \Vert{(1-R)\xo-(1-R)\x_t+R\xo-R\z}\Vert^2
			\nonumber\\
			&\leq 2(1-R)^2\Vert{\x_t-\xo}\Vert^2 + 2R^2\DFO^2.
		\end{align}
		
		Plugging-in \eqref{eq:2 linlem},\eqref{eq:3 linlem} into \eqref{eq:1 linlem}, we have
		{
			\begin{align*}
				f(\x_{t+1}) &\leq f(\x_t) + \eta(\x^*-\x_t)^{\top}\nabla{}f(\x_t) \\
				&+ \eta^2\beta\left({2R^2\Vert{\x_t-\x^*}\Vert^2 + 2R^2\DFO^2 + 2(1-R)^2\Vert{\x_t-\xo}\Vert^2 + 2R^2\DFO^2}\right) \\
				&\leq f(\x_t) + \eta(\x^*-\x_t)^{\top}\nabla{}f(\x_t) + \eta^2\beta\left({2\Vert{\x_t-\x^*}\Vert^2 + 4R^2\DFO^2}\right),
		\end{align*}}
		where we have used the fact that $\max_{R\in[0,1]} R^2+(1-R)^2=1$.\\
		Finally, using the gradient inequality and subtracting $f^*$ from both sides we get the proof.
	\end{proof}

	\section{Experiments}\label{sec: experiments}
	In this section we present numerical evidence that demonstrate the benefits of our NEP oracle-based algorithms. Additional results are deferred to Appendix~\ref{sec: apendexpriments}. 
	
	We conducted two experiments. In the first experiment we consider minimizing a random least-squares objective over the unit hypercube $[0,1]^d$. In the second experiment we consider the task of \textit{video co-localization} taken from \cite{lacoste2015global} which takes the form of minimizing a convex quadratic objective over the flow polytope. Note that in both experiments the feasible set is a convex and compact polytope. In both experiments we compare the performances of the original Frank-Wolfe algorithm (FW), the Frank-Wolfe with away-steps (AFW) variant \cite{lacoste2015global},  the fully-corrective Frank-Wolfe variant (FC) \cite{Jaggi13b}\footnote{Fully-corrective Frank-Wolfe is equivalent to using the fully-corrective option in Algorithm~\ref{alg:Algorithm 3} with $\rho_t=0$ i.e. with a linear optimization oracle instead of the NEP oracle.}, our Algorithm~\ref{alg:Algorithm 1}  (NEP FW),  and our Algorithm~\ref{alg:Algorithm 3} with the fully-corrective option (NEP FC). For the video co-localization experiment we also included the pairwise Frank-Wolfe variant (PFW)  \cite{lacoste2015global} and the DICG Frank-Wolfe variant (with line-search) \cite{GM16}.
	
	In both experiments, when implementing the fully corrective variants (FC, NEP FC), we used a constant number of FISTA \cite{beck2009fast} iterations in order to compute the next iterate (i.e., finding an approximate solution to the problem in line 9 of Algorithm~\ref{alg:Algorithm 3}). For our NEP oracle-based algorithms which require the smoothness parameter $\beta$, we set it precisely according to the data (i.e., largest eigenvalue of the Hessian, which is fixed since the objectives are quadratic).
	
	\subsection{Hypercube-constrained least-squares}
	We consider the problem $ \min_{\x\in[0,1]^d} \frac{1}{2}\Vert{\A\x-\b}\Vert^2$, where we take $\A$ to be a $m\times d$ matrix, $m=175, d=200$, with standard Gaussian entries and we set $\b=\A\x^*$, where $\xo$ is constructed by first choosing a random vertex of the hypercube and then changing it's first $5$ entries to $0.5$. Thus, $\xo$, which is also an optimal solution, lies on a face of dimension $5$ of the hypercube. The initialization point for all algorithms is taken to be $\vec{0}$. For both standard Frank-Wolfe and Algorithm~\ref{alg:Algorithm 1} we used the theoretical step-size $\frac{2}{t+1}$ (an alternative is to use line-search but for both variants it seems to give inferior results on this problem). For Algorithm \ref{alg:Algorithm 1} we skipped line 5 since it had no observable impact on performance. For Algorithm~\ref{alg:Algorithm 3}, on each iteration $t$ we took $\rho_t\in\{2^{a/4}\rho_{t-1}\}_{a=-4}^{4}$ which achieves the lowest function value, where initially we set $\rho_{0}=0.5$. For the FC and NEP FC variants, on each iteration $t$ we used 50 iterations of FISTA to compute the next iterate $\x_{t+1}$. Also, the smoothness constant $L_t$ of the FISTA objective was chosen to be fixed throughout all iterations and its was empirically tuned, resulting in $2\cdot 10^4$ for Algorithm~\ref{alg:Algorithm 3} and $5\cdot 10^4$ for FC.
	
	Figure \ref{fig: cubels} shows the results averaged over 50 i.i.d. runs (where in each run we sample a fresh matrix $\A$ and an optimal solution $\x^*$).
	As can be seen, NEP FC significantly outperforms all other algorithms both with respect to the number of iterations and runtime. Also, it can be seen that the simple addition of the NEP oracle in the NEP FW method  leads to substantial improvement in performance compared to the standard Frank-Wolfe method. Moreover,  with respect to runtime,  NEP FW outperforms all other methods except for NEP FC.
	\begin{figure}[h!]
		\centering
		\begin{minipage}[b]{0.49\textwidth}
			\includegraphics[width=\textwidth]{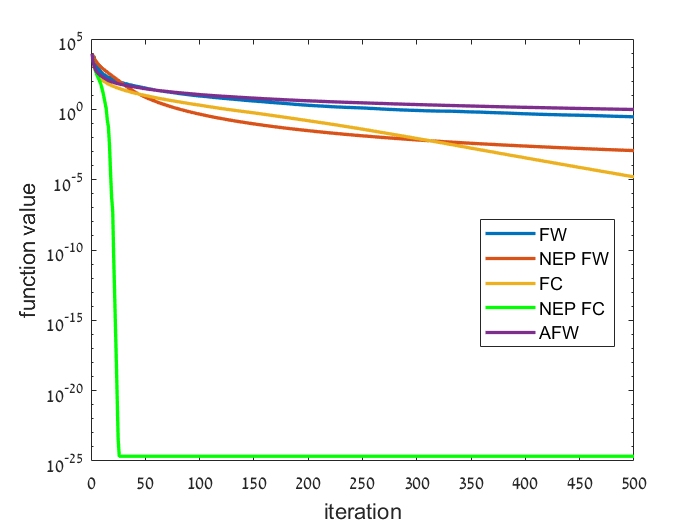}
		\end{minipage}
		\begin{minipage}[b]{0.49\textwidth}
			\includegraphics[width=\textwidth]{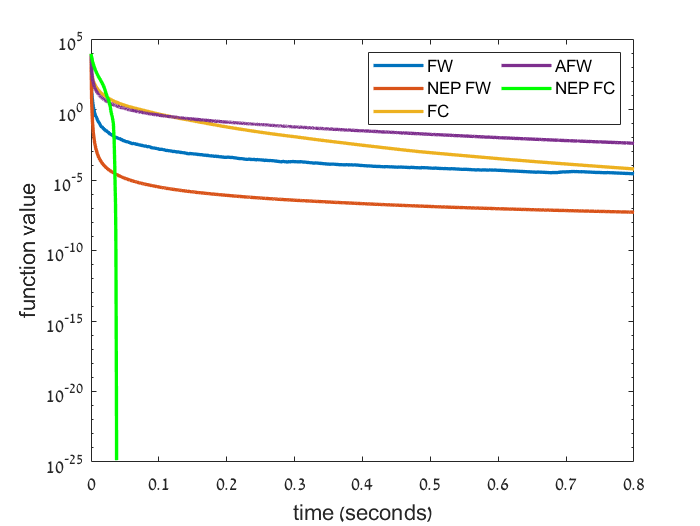}
		\end{minipage}
		\caption{Comparison of Frank-Wolfe variants on the hypercube-constrained least-squares problem. The results are the average of 50 i.i.d. runs.}
		\label{fig: cubels}
	\end{figure}
	\subsection{Video co-localization}
	For our second experiment we use a formulation of the \textit{video co-localization} task as a convex quadratic problem over the flow polytope (which is a 0--1 polytope), a formulation that was originally proposed in \cite{joulin2014efficient}. We  used the same dataset and initialization point used in \cite{lacoste2015global} and \cite{GM16}. The dimension of the problem is $d=660$ and the optimal solution has $66$ non-zero coordinates and no coordinate is equal to $1$, which implies that the optimal face is indeed low-dimensional.
	
	As opposed to the previous experiment, here for both the standard Frank-Wolfe method and Algorithm~\ref{alg:Algorithm 1} we used line-search to set the step-size since for both it gives better results than the fixed $\frac{2}{t+1}$ step-size. For Algorithm~\ref{alg:Algorithm 1} we use the theoretical constant $\eta_t=\frac{2}{t+1}$ for the regularization weight when calling the NEP oracle. For Algorithm~\ref{alg:Algorithm 3} we used $\rho_t=(1/\sqrt{2})^{t+1}$. For both FC and NEP FC variants, on each iteration $t$ we used 10 iterations of FISTA to compute $\x_{t+1}$, where as in the previous experiment, the FISTA smoothness parameter was fixed throughout all iterations and tuned empirically, resulting in a value of $0.25$ for both variants.
	
	Since the optimal value of the objective is not known we find it approximately  using 1000 iteration of DICG \cite{GM16} (which results in  a duality gap of $10^{-15}$).
	
	The results are given in Figure \ref{fig: videoco}. As it can be seen, NEP FC outperforms all other algorithms, both with respect to the number of iterations and running time. Although it may seem that the difference between NEP-FC and FC is not significant, the ratio between the time it takes  FC to reach an approximation error of $10^{-12}$ and the time it takes  NEP FC to reach the same error is $1.21$. Furthermore, it can be seen that the simple addition of the NEP oracle in the NEP FW method leads to substantial improvement in performance compared to the standard Frank-Wolfe method, and that with respect to running time, NEP FW outperforms the linearly-converging variants AFW and PFW.
	\begin{figure}[h!]
		\centering
		\begin{minipage}[b]{0.49\textwidth}
			\includegraphics[width=\textwidth]{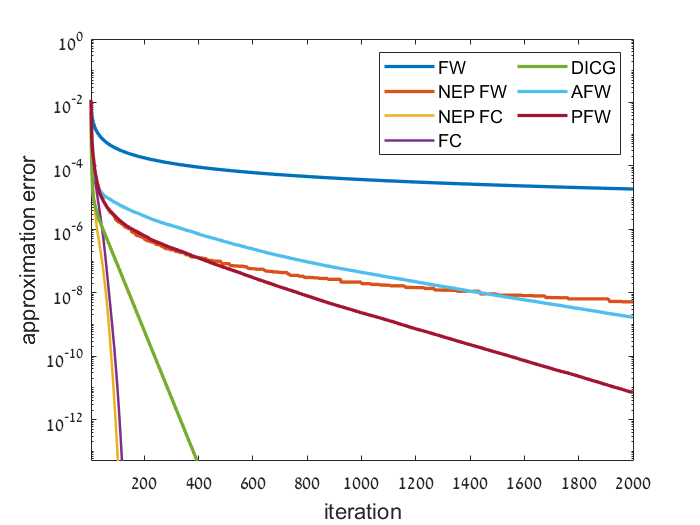}
		\end{minipage}
		\begin{minipage}[b]{0.49\textwidth}
			\includegraphics[width=\textwidth]{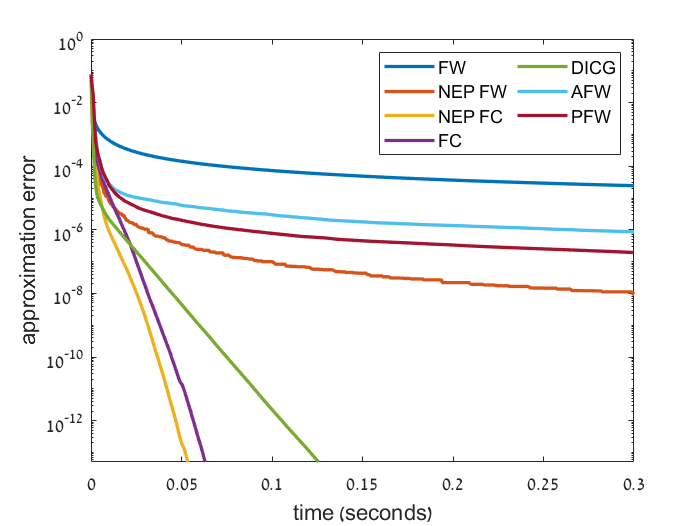}
		\end{minipage}
		\caption{Comparison of Frank-Wolfe variants on the video co-localization problem. The times shown are the averages of 200 runs.}
		\label{fig: videoco}
	\end{figure}
	
	\section*{Acknowledgments}
	This research was supported by the ISRAEL SCIENCE FOUNDATION (grant No. 1108/18).

	\bibliographystyle{unsrt}
	\bibliography{camera_readybib}
	\newpage
	\appendix
	
	\makeatletter
	\newtheorem*{rep@theorem}{\rep@title}
	\newcommand{\newreptheorem}[2]{%
		\newenvironment{rep#1}[1]{%
			\def\rep@title{#2 \ref{##1}}%
			\begin{rep@theorem}}%
			{\end{rep@theorem}}}
	\makeatother
	\newreptheorem{theorem}{Theorem}
	\section{Proof of Theorem~\ref{thm:alg1 rate}}
	For clarity, we first restate the theorem and then prove it.
	\begin{reptheorem}{thm:alg1 rate}
		Using Algorithm~\ref{alg:Algorithm 1} with step-size $\eta_t=\frac{2}{t+1}$ we have
		\begin{align*}
			\forall t\geq 2: \quad f(\x_t)-f^* \leq \frac{2\beta(\DO^2+D_L^2)}{t+1},
		\end{align*}
		where $D_L$ is the diameter of the initial level set (see Footnote \ref{footnote1}).
		
		Moreover, if $f(\cdot)$ has the quadratic growth property over $\mK$ with parameter $\alpha >0$, then 
		\begin{align*}
			\forall t\geq 2:\quad  f(\x_t)-f^*\leq \frac{2\beta\DO^2}{t+1} + \frac{\frac{8\beta^2}{\alpha}(\DO^2+\min\{2\alpha^{-1}(f(\x_1)-f^*),D_L^2\})\log(t)}{t^2}.
		\end{align*}
	\end{reptheorem} 
	\begin{proof}
		Using Lemma~\ref{lem: alg1 one step improve} with our choice of step size $\eta_t$, we have that for all $t\geq 1$,
		\begin{align*}
			f(\x_{t+1})-f^*\leq (1-\frac{2}{t+1})(f(\x_t)-f^*)+\frac{2\beta\DO^2}{(t+1)^2}+\frac{2\beta\dist(\x_t,\mX^*)^2}{(t+1)^2}.
		\end{align*}
		Thus, from Lemma~\ref{lem: new recursion} we have that for any $t\geq 2$,
		\begin{align}
			f(\x_t)-f^*&\leq \frac{1}{t^2}\sum_{k=1}^{t-1} 2\beta(\DO^2+\dist(\x_k,\mX^*)^2)\nonumber\\&\leq\nonumber
			\frac{\frac{1}{t-1}\sum_{k=1}^{t-1}2\beta\DO^2}{t+1} + \frac{1}{t^2}\sum_{k=1}^{t-1} 2\beta\dist(\x_k,\mX^*)^2\\&=
			\frac{2\beta\DO^2}{t+1} + \frac{1}{t^2}\sum_{k=1}^{t-1} 2\beta\dist(\x_k,\mX^*)^2\label{eq:l2use1}\\&\leq
			\frac{2\beta\DO^2}{t+1} + \frac{\frac{1}{t-1}\sum_{k=1}^{t-1}2\beta\dist(\x_k,\mX^*)^2}{t+1}
			\label{eq:l2use2}.
		\end{align}
		Since Algorithm~\ref{alg:Algorithm 1} is a decent method (i.e., the function value never increases from one iteration to the next), and so all iterates as well as the optimal set $\mX^*$ are contained within the initial level set $L$, for any $t\geq 1$ we can bound $\dist(\x_t,\mX^*)\leq D_L$.
		
		Thus, using \eqref{eq:l2use2} we have that,
		\begin{align*}
			\forall t\geq 2: \quad f(\x_t)-f^*\leq \frac{2\beta\DO^2}{t+1} + \frac{2\beta D_L^2}{t+1}.
		\end{align*}
		
		If we additionally assume quadratic growth, we can use again the fact that Algorithm~\ref{alg:Algorithm 1} is a decent method, in order bound $\dist(\x_t,\mX^*)^2\leq \frac{2}{\alpha}(f(\x_1)-f^*)$ which results in the bound
		\begin{align}\label{eq:thm3 fbound}
			\forall t\geq 2: \quad f(\x_t)-f^*\leq \frac{2\beta\DO^2}{t+1} + \frac{2\beta\min\{2\alpha^{-1}(f(\x_1)-f^*), D_L^2\}}{t+1}.
		\end{align}
		Thus, denoting $M=2\beta(\DO^2+\min\{\frac{2}{\alpha}(f(\x_1)-f^*), D_L^2\})$, using again the quadratic growth of $f(\cdot)$ and \eqref{eq:l2use1} we have that for any $t\geq 2$,
		\begin{align}
			f(\x_t)-f^*&\leq
			\frac{2\beta\DO^2}{t+1} + \frac{4\beta}{\alpha t^2}\sum_{k=1}^{t-1}(f(\x_k)-f^*)\underset{(a)}\leq
			\frac{2\beta\DO^2}{t+1} + \frac{4\beta}{\alpha t^2}\sum_{k=1}^{t-1}\frac{M}{k+1} \nonumber\\&=
			\frac{2\beta\DO^2}{t+1} +  \frac{\frac{4\beta}{\alpha}M(\sum_{k=1}^{t}\frac{1}{k}-1)}{t^2}\leq
			\frac{2\beta\DO^2}{t+1} + \frac{\frac{4\beta}{\alpha} M\log(t)}{t^2},\label{eq:thm3 lbound}
		\end{align}
		where (a) follows from the definition of $M$ and the bound in \eqref{eq:thm3 fbound}.
		
		Thus, plugging-in the value of $M$ in \eqref{eq:thm3 lbound} we indeed have that,
		\begin{align*}
			\forall t\geq 2:\quad f(\x_t)-f^*\leq
			\frac{2\beta\DO^2}{t+1} + \frac{\frac{8\beta^2}{\alpha}(\DO^2+\min\{2\alpha^{-1}(f(\x_1)-f^*), D_L^2\})\log(t)}{t^2}.
		\end{align*}
	\end{proof}
	\begin{lemma}\label{lem: new recursion}
		Let $\{a_t,b_t\}_{t=1}^\infty$ be non-negative scalars such that,
		\begin{align*}
			\forall t\geq 1: \quad a_{t+1}\leq (1-\frac{2}{t+1}) a_t + \frac{b_t}{(t+1)^2}.
		\end{align*}
		Then, we have that for any $t\geq 2$,
		\begin{align*}
			a_t\leq \frac{1}{t^2}\sum_{k=1}^{t-1} b_k.
		\end{align*}
		In particular, when $\{b_t\}_{t=1}^\infty$ is upper-bounded by $M>0$ we have that,
		\begin{align*}
			\forall t\geq 2:\quad a_t\leq\frac{M}{t+1}.
		\end{align*}
	\end{lemma}
	\begin{proof}
		First we define a sequence $\{\ha_t\}_{t=1}^\infty$ such that $\ha_1=a_1$ and
		$\forall t\geq1: \ha_{t+1}=(1-\frac{2}{t+1}) \ha_t + \frac{b_t}{(t+1)^2}$. Note that $\ha_2=\frac{b_1}{4}$ and that for any $t\geq 1$, $a_t\leq \ha_t$.\\
		Thus, since $\{\ha_t\}_{t=2}^{\infty}$ is of the form $\ha_{t+1}=c_t \ha_t+\bar{b}_t$ (a first-order non-homogeneous recurrence relation) we have that for any $t\geq2$,
		\begin{align*}
			\ha_t = \left(\prod_{i=2}^{t-1} c_i\right)\left(\ha_2 + \sum_{k=2}^{t-1} \frac{\bar{b}_k}{\prod_{i=2}^{k} c_i}\right).
		\end{align*}
		Thus, noting that for any $k\geq 2$,
		\begin{align*}
			\prod_{i=2}^{k} (1-\frac{2}{i+1})=\prod_{i=2}^{k} \frac{i-1}{i+1} = \frac{2}{k(k+1)},
		\end{align*}
		we have that for $t\geq 3$,
		\begin{align*}
			a_t&\leq \ha_t = \frac{2}{t(t-1)}\left(\ha_2 + \sum_{k=2}^{t-1} \frac{k(k+1)}{2}\cdot\frac{b_k}{(k+1)^2}\right)\\&\underset{(a)}=
			\frac{2}{t(t-1)}\sum_{k=1}^{t-1} \frac{k(k+1)}{2}\cdot\frac{b_k}{(k+1)^2}=
			\frac{1}{t(t-1)}\sum_{k=1}^{t-1} \frac{kb_k}{k+1}\\&\leq
			\frac{1}{t(t-1)}\cdot\frac{t-1}{t}\sum_{k=1}^{t-1} b_k=
			\frac{1}{t^2}\sum_{k=1}^{t-1} b_k,
		\end{align*}
		where (a) holds since $\ha_2=\frac{b_1}{4}$ (note that for that reason the bound also holds for $t=2$).
		
		Finally, if $\{b_t\}_{t=1}^\infty$ is upper-bounded by some $M>0$, we have that for any $t\geq 2$,
		\begin{align*}
			a_t \leq \frac{1}{t^2}\sum_{k=1}^{t-1} b_k \leq \frac{(t-1)M}{(t+1)(t-1)} = \frac{M}{t+1}.
		\end{align*}
	\end{proof}
	
	\section{Proof of Theorem~\ref{thm: theorem 1}}
	For clarity, we first restate the theorem and then prove it.
	\begin{reptheorem}{thm: theorem 1}
		Let $\mK=[0,1]^d$, and fix a positive integer $m<d$. Let $\xo=\frac{1}{2}\sum_{i=1}^{m} \e_i$, i.e., $\xo$ has $\frac{1}{2}$ for the first $m$ coordinates and 0 for the rest.
		Now consider the minimization of the function $f(\x) = \frac{1}{2}\lVert \x - \xo \rVert^2$ over $\mK$ starting at $\x_1=\e_{m+1}$. Then, for any Frank-Wolfe-type method there exists a sequence of answers returned by the linear optimization oracle such that for any $t\leq \lfloor \sqrt{d-m-1} \rfloor$, the $t$th iterate of the algorithm $\x_t$ satisfies $f(\x_t)-f^* \geq \frac{1}{4}$. 
	\end{reptheorem}
	\begin{proof}
		Clearly, the unique optimal solution is $\xo$ and $f^*=f(\xo)=0$. Let $k=\lfloor \sqrt{d-m-1} \rfloor$ and let $S_0, S_1\dots,S_k$ be a partition of the last $d-m-1$ coordinates (i.e., of the set $\{m+2\leq i\leq d:i\in\mathbb{N}\}$) such that each $S_i, i=1,\dots,k$ contains exactly $k$ coordinates. Consider now the iterates of some Frank-Wolfe-type method. Observe that for any $t\geq 1$, since the last $d-m$ coordinates of $\nabla f(\x_t)=\x_ t-\xo$ are the same as those of $\x_t$, the last $d-m$ coordinates of a valid answer returned by the linear optimization oracle can contain $0$ in coordinates in which $\x_t$ is non-zero, and either $0$ or $1$ in coordinates in which $\x_t$ is $0$. Thus, a valid sequence of answers returned by the linear optimization oracle on iterations $1 \leq t\leq k$ may set on each iteration $t\in\{1,\dots,k\}$ the oracle's output $\v_t$ to contain $1$ in the coordinates in $S_t$ and $0$ in the rest of the last $d-m$ coordinates (i.e., the coordinates in $(\bigcup_{i=0}^kS_i\setminus{}S_t)\cup\{m+1\}$).
		
		For any vector $\x\in\reals^d$ we let $\x^{(d-m)}$ denote the restriction of $\x$ to the last $d-m$ coordinates. 
		Now, fix some  $t\leq k+1$. Since $\x_t\in\conv(\x_1,\v_1,\dots,\v_k)$, there exist some $\w\in\conv(\v_1,\dots,\v_k)$ and $\gamma\in[0,1]$ such that $\x_t=(1-\gamma)\x_1+\gamma\w$. Note that since $\w^{(d-m)}$ is a convex combination of $k$ orthogonal vectors ($\v_1^{(d-m)},\dots,\v_k^{(d-m)}$), each having Euclidean norm $\sqrt{k}$, we have that $\Vert{\w^{(d-m)}}\Vert^2\geq 1$. Thus, using the fact that $\xo^{(d-m)}=\mathbf{0}$ and that $\w^{(d-m)}$ is orthogonal to $\x_1^{(d-m)}$, we have that,
		\begin{align*}
			f(\x_t)-f^*&= \frac{1}{2}\Vert{\x_t-\xo}\Vert^2\geq \frac{1}{2}\Vert{\x_t^{(d-m)}}\Vert^2 = \frac{1}{2}\Vert{(1-\gamma)\x_1^{(d-m)}+\gamma\w^{(d-m)}}\Vert^2\\&=
			\frac{1}{2}(1-\gamma)^2\Vert{\x_1^{(d-m)}}\Vert^2+\frac{1}{2}\gamma^2\Vert{\w^{(d-m)}}\Vert^2\geq \frac{1}{2}(1-\gamma)^2+\frac{1}{2}\gamma^2\geq \frac{1}{4},
		\end{align*}
		where the last inequality holds since $\min_{\gamma\in[0,1]}(1-\gamma)^2+\gamma^2=\frac{1}{2}$.
		
		Thus, we indeed have that for any $t\leq k=\lfloor \sqrt{d-m-1} \rfloor$, $f(\x_t)-f^*\geq \frac{1}{4}$.
	\end{proof}

	\section{Results and proofs missing from Section~\ref{sec:linres}}
	In Section \ref{sec:linrate alg3} we prove Theorem \ref{thm: linrate alg3}, in Section \ref{sec:adaptiveLinear} we prove that the linear convergence rate in Theorem \ref{thm: linrate alg3} also holds if instead of using a predefined sequence $\{\rho_t\}_{t\geq 1}$, we use an adaptive-step size strategy, and in Section \ref{sec:proofOfLinDelta} we prove Theorem \ref{thm: alg3 lin delta}.
	
	\subsection{Proof of Theorem~\ref{thm: linrate alg3}}\label{sec:linrate alg3}
	We first prove the following technical observation which was used in the proof of Lemma~\ref{lem: one step improve alg3} and then prove the theorem.
	\begin{observation}\label{obs:linlem help}
		Suppose that $\mK$ is a convex and compact polytope and let $\x\in\mK$ which is given by a convex combination of vertices $\x=\sum_{i=1}^k \lambda_i\v_i$. Suppose there exists $R\leq 1$ such that some $\x^*\in\mX^*$ can be written as $\xo=\sum_{i=1}^k(\lambda_i-\Delta^*_i)\v_i + \sum_{i=1}^k\Delta^*_i\z$ for $\Delta^*_i\in[0,\lambda_i], \z\in\FO$ and $\sum_{i=1}^k\Delta_i^* \leq R$. Then, $\xo$ can also be written as, $\xo=\sum_{i=1}^k (\lambda_i-\bar{\Delta}_i)\v_i + \sum_{i=1}^k\bar{\Delta}_i\bar{\z}$ where, $\bar{\Delta}_i\in[0,\lambda_i], \bar{\z}\in\FO$ and $\sum_{i=1}^k\bar{\Delta}_i = R$. 
	\end{observation}
	\begin{proof}
		Denote $R'=\sum_{i=1}^k\Delta^*_i$ and assume $R' < R$. Let $\gamma=\frac{R-R'}{1-R'}$ and note that since $R'< R\leq 1$, $\gamma\in[0,1]$. We will show that $\bar{\Delta}_i=(1-\gamma)\Delta_i^*+\gamma\lambda_i$ and $\bar{\z}=\frac{(1-\gamma)R'}{R}\z+\frac{\gamma}{R}\xo$ satisfy the required conditions.
		
		Indeed, since $\gamma\in[0,1]$, we have that $\bar{\Delta}_i\in[0,\lambda_i]$, and since
		\begin{align}\label{eq:obs2 1}
			(1-\gamma)R'+\gamma=\frac{R'(1-R)}{1-R'}+\frac{R-R'}{1-R'}=R,
		\end{align}
		we have that, $\bar{\z} = \frac{(1-\gamma)R'}{R}\z+\frac{\gamma}{R}\xo$ is a convex combination of points in $\FO$ and thus in $\FO$ itself.
		
		Moreover, using \eqref{eq:obs2 1} we have that,
		\begin{align*}
			\sum_{i=1}^k \bar{\Delta}_i &= (1-\gamma)\sum_{i=1}^k \Delta_i^* + \gamma\sum_{i=1}^k \lambda_i = (1-\gamma)R'+\gamma=R.
		\end{align*} 
		Finally, we indeed have that,
		\begin{align*}
			\xo &= (1-\gamma)\xo + \gamma\xo= \sum_{i=1}^k (1-\gamma)(\lambda_i-\Delta_i^*)\v_i+(1-\gamma)R'\z+\gamma\xo \\&=
			\sum_{i=1}^k (\lambda_i-\bar{\Delta}_i)\v_i + R\bar{\z},
		\end{align*}
		completing the proof.
	\end{proof}
	
	Before continuing to the proof of Theorem~\ref{thm: linrate alg3} we first restate it.
	\begin{reptheorem}{thm: linrate alg3}
		Suppose that $\mK$ is a convex and compact polytope and quadratic growth holds with parameter $\alpha > 0$. Let $C\geq f(\x_1)-f^*$ and $M\geq\max\{\frac{\beta}{\alpha}(4+8d\mu^2\DFO^2),\frac{1}{2}\}$, where $\mu=\frac{\psi}{\xi}$. Using Algorithm~\ref{alg:Algorithm 3} with parameter $\rho_t=\frac{\min\{\sqrt{\frac{2Cd\mu^2}{\alpha}\exp\big({-\frac{1}{4M}(t-1)}\big)},1\}}{2M}$  for all $t\geq 1$
		one has,
		\begin{align*}
			\forall t\geq1: \quad f(\x_t)-f^*\leq C\exp\Big(-\frac{t-1}{4M}\Big),
		\end{align*} 
	\end{reptheorem}
	\begin{proof}
		The proof is by induction on $t$. For $t=1$ the bound holds by the definition of the constant $C$.
		
		Now suppose the bound holds for some $t\geq 1$. We will show that it holds for $t+1$. Let $\x_t^*\in \mX^*$ be the closest optimal solution to $\x_t$ and let $\sum_{i=1}^{k}\lambda_i\v_i$ be the decomposition of $\x_t$ used in Algorithm~\ref{alg:Algorithm 3}.\\
		By Lemma 5.5 from \cite{Garber16linearly} there exist $\Delta_i^*\in[0,\lambda_i]$, $i=1,\dots,k$ and $\z\in\FO$, such that $\x_t^*=\sum_{i=1}^k (\lambda_i-\Delta_i^*)\v_i+\sum_{i=1}^k \Delta_i^*\z$, and $\sum_{i=1}^k \Delta_i^*\leq\min\{\sqrt{d}\mu\Vert{\x_t-\x_t^*}\Vert,1\}$, which by the quadratic growth of $f(\cdot)$ together with the induction assumption, implies that $\sum_{i=1}^k \Delta_i^*\leq \min\{\sqrt{\frac{2Cd\mu^2}{\alpha}\exp(-\frac{1}{4M}(t-1))},1\}$, where $M$ is as defined in the theorem. 
		
		Therefore, taking $R=\min\{\sqrt{\frac{2Cd\mu^2}{\alpha}\exp(-\frac{1}{4M}(t-1))},1\}$ and $\eta=\frac{1}{2M}$, we have that $\rho_t=\eta R$ and $\sum_{i=1}^k \Delta_i^*\leq R$ and thus, we can use Lemma~\ref{lem: one step improve alg3} which implies that,
		\begin{align*}
			f(\x_{t+1})-f^*&\leq (1-\eta)(f(\x_t)-f^*)+\eta^2\beta(2\Vert{\x_t-\x_t^*}\Vert^2+4R^2\DFO^2)\\&\underset{(a)}\leq
			(1-\eta)Ce^{-\frac{t-1}{4M}}+\eta^2\beta(\frac{4C}{\alpha}e^{-\frac{t-1}{4M}}+\frac{8Cd\mu^2\DFO^2}{\alpha}e^{-\frac{t-1}{4M}})\\&\underset{(b)}\leq
			(1-\eta)Ce^{-\frac{1}{4M}(t-1)}+M\eta^2Ce^{-\frac{1}{4M}(t-1)}\\&\underset{(c)}=
			(1-\frac{1}{4M})Ce^{-\frac{1}{4M}(t-1)}\underset{(d)}\leq Ce^{-\frac{1}{4M}t},
		\end{align*}
		where (a) holds by the induction assumption and the quadratic growth of $f(\cdot)$, (b) holds due to the definition of $M$, (c) holds since $\eta=\frac{1}{2M}$, and (d) holds since $1-x\leq e^{-x}$.
	\end{proof}
	
	\subsection{Linear convergence with adaptive step-sizes}\label{sec:adaptiveLinear}
	We now prove that, in principle, the linear rate of Theorem \ref{thm: linrate alg3} can be achieved with an adaptive choice of the parameter $\rho_t$, instead of the predefined value listed in Theorem \ref{thm: linrate alg3}. Theorem \ref{thm:alg3 param} demonstrates that it suffices to do a log-scale search over $\rho_t$, i.e., check values $\rho_t = 1, \frac{1}{2}, \frac{1}{4}, \frac{1}{8},\dots$, and take the one which leads to the largest decrease in function value. Note that according to the theorem and since $\mK$ is compact, if the target accuracy we are looking to obtain is some $\epsilon >0$, we need not consider values of $\rho_t$ below some $O(\epsilon)$. Thus, the overall number of search steps will be logarithmic in $1/\epsilon$, and the overall increase in complexity due to the use of such adaptive step-sizes will be an $O(\log{1/\epsilon})$ factor. 
	
	\begin{theorem}\label{thm:alg3 param}
		Suppose that all of the assumptions of Theorem~\ref{thm: linrate alg3} hold and denote $h_1=f(\x_1)-f^*$ and $M^*=\max\{\frac{\beta}{\alpha}(4+8d\mu^2\DFO^2),\frac{1}{2}\}$. Consider running Algorithm~\ref{alg:Algorithm 3} and for all $t\geq 1$ define $\rho_t^*=\frac{\min\{\sqrt{d}\mu\cdot\dist(\x_t,\mX^*),1\}}{2M^*}$. Suppose that for all $t\geq 1$ we use some value $\rho_t\in[\frac{\rho_t^*}{2},\rho_t^*]$. Then, with these parameters one has,
		\begin{align*}
			\forall t\geq1: f(\x_t)-f^*\leq h_1\exp\Big(-\frac{3(t-1)}{16M^*}\Big).
		\end{align*} 
	\end{theorem}
	\begin{proof}
		Fix some iteration $t\geq 1$. Using the same notation and arguments as in the proof of Theorem~\ref{thm: linrate alg3}, by Lemma 5.5 from \cite{Garber16linearly} we can take $R=\min\{\sqrt{d}\mu\Vert{\x_t-\x_t^*}\Vert,1\}$ (note that $\Vert{\x_t-\x_t^*}\Vert=\dist(\x_t,\mX^*)$ as $\x_t^*$ denotes the closest optimal solution to $\x_t$). Thus, taking $\eta=\rho_t/R$ and noting that by the definition of $\rho_t$ in the theorem, $\eta\in[\frac{1}{4M^*},\frac{1}{2M^*}]$, by Lemma~\ref{lem: one step improve alg3} we have that,
		\begin{align*}
			f(\x_{t+1})-f^*&\leq (1-\eta)(f(\x_t)-f^*)+\eta^2\beta(2\Vert{\x_t-\x_t^*}\Vert^2+4R^2\DFO^2)\\&\underset{(a)}\leq
			(1-\eta)(f(\x_t)-f^*) + \eta^2\beta(\frac{4}{\alpha}+\frac{8d\mu^2\DFO^2}{\alpha})(f(\x_t)-f^*)
			\\&\underset{(b)}=(1-\eta+M^*\eta^2)(f(\x_t)-f^*)\underset{(c)}\leq (1-\frac{3}{16M^*})(f(\x_t)-f^*),
		\end{align*}
		where (a) follows from the quadratic growth of $f(\cdot)$ and the definition of $R$, (b) follows from the definition of $M^*$, and (c) holds since $\eta\in[\frac{1}{4M^*},\frac{1}{2M^*}]$.
		
		Using the fact that $1-x\leq e^{-x}$, we have that,
		\begin{align*}
			f(\x_t)-f^*\leq h_1(1-\frac{3}{16M^*})^{t-1}\leq h_1\exp\Big(-\frac{3(t-1)}{16M^*}\Big).
		\end{align*} 
	\end{proof}
	\subsection{Proof of Theorem~\ref{thm: alg3 lin delta}}\label{sec:proofOfLinDelta}
	The proof goes along the same lines as the proof of Theorem \ref{thm: linrate alg3}, but this time, since  we assume $\delta$-strict complementarity, we can use Lemma 2 from \cite{Garber20} instead of Lemma 5.5 from \cite{Garber16linearly} in order to upper-bound the amount of probability mass we need to move from the convex decomposition of the point $\x_t$ in order to reach the closest optimal solution $\xo$, yielding a dimension-independent linear convergence rate.
	
	For clarity, before continuing to the proof of the theorem we first restate it.
	\begin{reptheorem}{thm: alg3 lin delta}
		Suppose that in addition to the assumptions of Theorem \ref{thm: linrate alg3}, Assumption~\ref{asm: dstrict} also holds with some parameter $\delta>0$, and let $C\geq f(\x_1)-f^*$, $M_1\geq\max\{\frac{4\beta}{\alpha}+8\beta\DFO^2\max\{2\kappa,\delta^{-1}\},\frac{1}{2}\}$ and  $M_2\geq\{\frac{4\beta}{\alpha}+16\beta\kappa\DFO^2,\frac{1}{2}\}$, where $\kappa=\frac{2\mu^2\dim\FO}{\alpha}$.\\
		Using Algorithm~\ref{alg:Algorithm 3} with parameters $\rho_t=\frac{\min\{\sqrt{2\max\{2\kappa,\delta^{-1}\}C\exp\big(-\frac{1}{4M_1}(t-1)\big)},1\}}{2M_1}$ for all $t\geq 1$, one has,
		\begin{align}
			\forall t\geq 1:\quad f(\x_t)-f^*\leq C\exp\Big(-\frac{t-1}{4M_1}\Big). \tag{\ref{eq: alg3 lin delta res 1}}
		\end{align}
		Furthermore, in case that $2\kappa\leq\delta^{-1}$ and $\dim\FO>0$ (i.e. $\kappa>0$), denoting $\tau\geq 4M_1\log(\frac{C}{\delta^2\kappa})+1$ and for all $t\geq \tau$, using $\rho_t=\frac{\min\{2\delta\kappa \exp(-\frac{t-\tau}{8M_2}),1\}}{2M_2}$ instead of the above value, one has,
		\begin{align*}
			\forall t\geq \tau:\quad f(\x_t)-f^*\leq\delta^2\kappa \exp\Big(-\frac{t-\tau}{4M_2}\Big). \tag{\ref{eq: alg3 lin delta res 2}}
		\end{align*}
	\end{reptheorem}
	\begin{proof}
		We first prove the rate in \eqref{eq: alg3 lin delta res 1}  by induction on $t$. For $t=1$ the bound holds by the definition of the constant $C$.
		
		Now, suppose the bound holds for some $t\geq 1$. We will show that it holds for $t+1$. Let $\x_t^*\in \mX^*$ be the closest optimal solution to $\x_t=\sum_{i=1}^{k}\lambda_i\v_i$ and denote $h_t=f(\x_t)-f^*$. Then, by Lemma 2 from \cite{Garber20}, there exist $\Delta_i^*\in[0,\lambda_i], i=1,\dots,k$, and $\z\in\FO$, such that $\x_t^*=\sum_{i=1}^k (\lambda_i-\Delta_i^*)\v_i+\sum_{i=1}^k \Delta_i^*\z$ and $\sum_{i=1}^k \Delta_i^*\leq \min\{1, \delta^{-1}h_t+\sqrt{\kappa h_t}\}$.
		
		Since when $\delta^{-1}h_t > \sqrt{\kappa h_t}$ we have that $\sum_{i=1}^k \Delta_i^*\leq \min\{1,2\delta^{-1}h_t\}\leq\min\{1,\delta^{-1/2}\sqrt{2h_t}\}$, and otherwise we have that $\sum_{i=1}^k \Delta_i^*\leq \min\{1,2\sqrt{\kappa h_t}\}$, by taking the maximum of the two we have that $\sum_{i=1}^k \Delta_i^*\leq\min\{1,\max\{\sqrt{2\kappa},\delta^{-1/2}\}\sqrt{2h_t}\}$. Thus, taking $\eta=\frac{1}{2M_1}$ and $R=\min\{1,\sqrt{2\max\{2\kappa,\delta^{-1}\}C\exp({-\frac{t-1}{4M_1}})}\}$, where $M_1$ is as defined in the theorem, we have that $\rho_t=\eta R$ and $\sum_{i=1}^k \Delta_i^*\leq R$, and therefore we can use Lemma~\ref{lem: one step improve alg3} which implies that,
		\begin{align*}
			h_{t+1}&\leq (1-\eta)h_t+\eta^2\beta(2\Vert{\x_t-\x_t^*}\Vert^2+4R^2\DFO^2)\\&\underset{(a)}\leq
			(1-\eta)Ce^{-\frac{1}{4M_1}(t-1)}+\eta^2\beta(\frac{4}{\alpha}Ce^{-\frac{1}{4M_1}(t-1)}+8\max\{2\kappa,\delta^{-1}\}\DFO^2 Ce^{-\frac{1}{4M_1}(t-1)})\\&\underset{(b)}\leq
			(1-\eta)Ce^{-\frac{1}{4M_1}(t-1)} + M_1\eta^2 Ce^{-\frac{1}{4M_1}(t-1)}\\&\underset{(c)}=
			(1-\frac{1}{4M_1})Ce^{-\frac{1}{4M_1}(t-1)}\underset{(d)}\leq e^{-\frac{1}{4M_1}t},
		\end{align*}
		where (a) holds by the induction assumption and the quadratic growth of $f(\cdot)$, (b) holds due to the definition of $M_1$, (c) holds since $\eta=\frac{1}{2M_1}$, and (d) holds since $1-x\leq e^{-x}$.
		
		Now we turn to prove the rate in \eqref{eq: alg3 lin delta res 2}. The proof goes along the same lines as the proof of \eqref{eq: alg3 lin delta res 1}, but now we have  from \eqref{eq: alg3 lin delta res 1} that for $\tau$ as defined in the theorem, it holds that $h_{\tau}\leq \delta^2\kappa$, and since when $h_t\leq \delta^2\kappa$ we have that $\sum_{i=1}^k \Delta_i^*\leq \min\{1, \delta^{-1}h_t+\sqrt{\kappa h_t}\}\leq \min\{1, 2\sqrt{\kappa h_t}\}$, we can replace $C$ with $\delta^2\kappa$ and $\max\{2\kappa,\delta^{-1}\}$ with $2\kappa$ in the above arguments, and get the desired bound for all $t\geq\tau$.
	\end{proof}
	
	\section{Results and proofs missing from Section~\ref{sec:stoc results}}
	Our NEP Oracle-based Stochastic Frank-Wolfe variant is given in Algorithm \ref{alg:Algorithm stochastic}
	\begin{algorithm}
		\caption{Stochastic Frank-Wolfe with a Nearest Extreme Point Oracle}
		\label{alg:Algorithm stochastic}
		\begin{algorithmic}[1]
			\STATE Input: a sequence of step-sizes $\{\eta_t\} \subset [0,1]$, a sequence of mini-batch sizes $\{m_t\}\subset\mathbb{N}$
			\STATE $\x_1 \leftarrow$ some arbitrary point in $\mV$
			\FOR {$t=1 \dotsc $}
			\STATE Compute $\tnbt$ as the average of $m_t$ iid unbiased estimators of $\nabla f(\x_t)$
			\STATE $\v_t \leftarrow \argmin_{\v\in \mV} \v^{\top}\tnbt + \frac{\beta\eta_t}{2}\lVert \x_t - \v \rVert^2$
			\STATE $\x_{t+1} \leftarrow (1-\eta_t)\x_t + \eta_t \v_t$
			\ENDFOR
		\end{algorithmic}
	\end{algorithm}
	
	In Section \ref{sec:stochasticRateProof} we prove a theorem on the convergence rate of Algorithm \ref{alg:Algorithm stochastic}. Then, in Section \ref{sec:stochasticComplexProof} we prove Theorem \ref{thm:stoc eps convergence}.

	Throughout this section, for all $t\geq 1$ we denote $\hbar_t=\E[f(\x_t)-f^*]$.
	
	\subsection{Convergence rate of Algorithm~\ref{alg:Algorithm stochastic}}\label{sec:stochasticRateProof}
	\begin{theorem}\label{thm: stoc improved rate} 
		Assume $f(\cdot),\mK$ satisfy the quadratic growth property with parameter $\alpha >0$. Then, using Algorithm~\ref{alg:Algorithm stochastic} with step-size $\eta_t=\frac{2}{t+1}$ and mini-batch sizes that satisfy
		\begin{align*}
			m_t \geq \max\left\{\left(\frac{G(t+1)}{\beta\DK}\right)^2,\min\left\{\left(\frac{G\DK(t+1)}{\beta\DO^2}\right)^2,\left(\frac{\alpha G(t+1)^2}{8\beta^2 \DK}\right)^2\right\}\right\},
		\end{align*}
		one has,
		\begin{align*}
			\forall t\geq 2: \quad \E[f(\x_t)-f^*] \leq \frac{4\beta{}D^{*2}}{t+1}+\frac{\frac{32\beta^2}{\alpha}{}D_{\mK}^2\log(t)}{t^2}.
		\end{align*}
	\end{theorem}
	We first prove a lemma on the improvement (in expectation) on each iteration of Algorithm~\ref{alg:Algorithm stochastic}, and then we prove the theorem.
	\begin{lemma}\label{lem: stoc one iter improve}
		Fix some iteration $t\geq 1$ of Algorithm~\ref{alg:Algorithm stochastic} and let  $d_t\geq \E[\dist(\x_t,\mX^*)^2]$. Then, using a step-size $\eta_t\in [0,1]$ and mini-batch size $m_t\geq(\frac{2G D_{\mK}}{\beta\eta_t\min\{\DO^2+d_t,\DK^2\}})^2$ one has,
		\begin{align*}
			\hbar_{t+1}\leq (1-\eta_t)\hbar_t + \beta\eta_t^2(\min\{\DO^2+d_t,\DK^2\}).
		\end{align*}
	\end{lemma}
	\begin{proof}
		Let $\xo$ be the closest optimal solution to $\x_t$. We will upper-bound $\tnbt^\top\v_t+\frac{\beta\eta_t}{2}\Vert{\v_t-\x_t}\Vert^2$ in two ways.\\
		On one hand, consider an extreme point $\u^*\in{\argmin}_{\v\in\mV}\tnbt^\top\v$. By the optimality of $\v_t$ we have that,
		\begin{align*}
			\tnbt^\top\v_t+\frac{\beta\eta_t}{2}\Vert{\v_t-\x_t}\Vert^2\leq \tnbt^\top\u^*+\frac{\beta\eta_t}{2}\Vert{\u^*-\x_t}\Vert^2\leq
			\tnbt^\top\xo+\frac{\beta\eta_t\DK^2}{2}. 
		\end{align*}
		On the other hand, since $\xo$ is a convex combination of extreme points from the set $S^*$, as in the proof of Lemma~\ref{lem: alg1 one step improve},
		there must exist some $\w^* \in S^*$ such that,
		\begin{align*}
			\tnbt^\top\w^*+\beta\eta_t(\w^*-\xo)^\top(\xo-\x_t)\leq \tnbt^\top\xo.
		\end{align*}
		Thus, denoting $\dbar_t=\dist(\x_t,\mX^*)$, by the optimality of $\v_t$ we have that,
		\begin{align*}
			&\tnbt^\top\v_t+\frac{\beta\eta_t}{2}\Vert{\v_t-\x_t}\Vert^2\leq \tnbt^\top\w^*+\frac{\beta\eta_t}{2}\Vert{\w^*-\x_t}\Vert^2=\\&\tnbt^\top\w^*+\frac{\beta\eta_t}{2}\Vert{\w^*-\xo+\xo-\x_t}\Vert^2=\\&
			\tnbt^\top\w^* + \frac{\beta\eta_t}{2}\Vert{\w^*-\xo}\Vert^2+\beta\eta_t(\w^*-\xo)^\top(\xo-\x_t)+\frac{\beta\eta_t}{2}\Vert{\x_t-\xo}\Vert^2\leq\\&\tnbt^\top\xo + \frac{\beta\eta_t\DO^2}{2} + \frac{\beta\eta_t \dbar_t^2}{2}.
		\end{align*}
		Thus, denoting $\Mbar_t=\min\{\DO^2+\dbar_t^2,\DK^2\}$, we have that,
		\begin{align}\label{eq: stoc v_t bound}
			\tnbt^\top\v_t+\frac{\beta\eta_t}{2}\Vert{\v_t-\x_t}\Vert^2\leq \tnbt^\top\xo + \frac{\beta\eta_t\Mbar_t}{2}.
		\end{align}
		Now, from the $\beta$-smoothness of $f(\cdot)$ we have that,
		\begin{align*}
			f(\x_{t+1})-f(\x_t)&\leq \nabla f(\x_t)^\top(\x_{t+1}-\x_t) + \frac{\beta}{2}\Vert \x_{t+1}-\x_t\Vert^2 \\&=
			\eta_t\nabla f(\x_t)^\top(\v_t-\x_t)+\frac{\beta\eta_t^2}{2}\lVert \v_t-\x_t\rVert^2\\&=\eta_t(\nabla f(\x_t)-\tnbt)^\top(\v_t-\x_t) + \eta_t\tnbt^\top(\v_t-\x_t) + \frac{\beta\eta_t^2}{2}\Vert \v_t-\x_t\Vert^2\\&\underset{(a)}\leq
			\eta_t(\nabla f(\x_t)-\tnbt)^\top(\v_t-\x_t) + \eta_t\tnbt^\top(\xo-\x_t) + \frac{\beta\eta_t^2\Mbar_t}{2}\\&=
			\eta_t\nabla f(\x_t)^\top(\xo-\x_t) + \eta_t(\nabla f(\x_t)-\tnbt)^\top(\v_t-\xo) + \frac{\beta\eta_t^2\Mbar_t}{2}\\&\underset{(b)}\leq
			\eta_t(f(\xo)-f(\x_t))+\eta_t D_{\mK}\Vert{\tnbt-\nabla f(\x_t)}\Vert + \frac{\beta\eta_t^2\Mbar_t}{2},
		\end{align*}
		where (a) follows from \eqref{eq: stoc v_t bound}, and (b) follows from the convexity of $f(\cdot)$ and the Cauchy-Schwarz inequality.\\
		Now,  using Jensen's inequality we have that $\E[\Vert{\tnbt-\nabla f(\x_t)}\Vert]\leq \sqrt{\E[\Vert{\tnbt-\nabla f(\x_t)}\Vert^2]}\leq \frac{G}{\sqrt{m_t}}$, which by our choice of $m_t$, is at most $\frac{\beta\eta_t\min\{\DO^2+d_t,\DK^2\}}{2D_{\mK}}$. Thus, taking expectation and noting that $\E[\Mbar_t]=\E[\min\{\DO^2+\dbar_t^2,\DK^2\}]\leq {\min\{\DO^2+d_t,\DK^2\}}$, we have that
		\begin{align*}
			\hbar_{t+1}-\hbar_t\leq -\eta_t \hbar_t + \beta\eta_t^2(\min\{\DO^2+d_t,\DK^2\}).
		\end{align*}
		Finally, rearranging, we have that
		\begin{align*}
			\hbar_{t+1}\leq (1-\eta_t)\hbar_t + \beta\eta_t^2(\min\{\DO^2+d_t,\DK^2\}).
		\end{align*}
	\end{proof}
	
	\begin{proof}[Proof of Theorem~\ref{thm: stoc improved rate}]
		First note that for any $t\geq 1$,
		\begin{align*}
			m_t\geq \left(\frac{G(t+1)}{\beta\DK}\right)^2=\left(\frac{2G\DK}{\beta\eta_t\min\{\DK^2, \DO^2+\DK^2\}}\right)^2.
		\end{align*}
		Thus, we can use Lemma~\ref{lem: stoc one iter improve} with the trivial bound $\DK^2\geq\E[\dist(\x_t,\mX^*)^2]$ which, with our choice of step size $\eta_t$, implies that,
		\begin{align*}
			\forall t\geq1: \quad \hbar_{t+1}\leq (1-\frac{2}{t+1})\hbar_t + \frac{4\beta\DK^2}{(t+1)^2},
		\end{align*}
		which by Lemma~\ref{lem: new recursion} (with $M=4\beta\DK^2$) gives us the bound,
		\begin{align}\label{eq:stoc SFW bound}
			\forall t\geq 2:\quad  \hbar_t\leq \frac{4\beta\DK^2}{t+1}.
		\end{align}
		Now, for any $t\geq 1$, let $d_t=\frac{8\beta\DK^2}{\alpha(t+1)}$. By the quadratic growth of $f(\cdot)$ and \eqref{eq:stoc SFW bound}, we have that $\E[\dist(\x_t,\mX^*)^2]\leq d_t$. Thus, since the mini-batch sizes satisfy for all $t\geq 1$,
		\begin{align*}
			m_t&\geq \max\left\{\left(\frac{G(t+1)}{\beta\DK}\right)^2,\min\left\{\left(\frac{G\DK(t+1)}{\beta\DO^2}\right)^2,\left(\frac{\alpha G(t+1)^2}{8\beta^2 \DK}\right)^2\right\}\right\}\\&=
			\max\left\{\left(\frac{2G\DK}{\beta\eta_t\DK^2}\right)^2,
			\min\left\{\left(\frac{2G\DK}{\beta\eta_t\DO^2}\right)^2,\left(\frac{ 2G\DK}{\beta\eta_t d_t}\right)^2\right\}\right\}\\&=
			\left(\frac{2G\DK}{\beta\eta_t\min\{\DK^2,\max\{\DO^2,d_t\}\}}\right)^2\geq
			\left(\frac{2G\DK}{\beta\eta_t\min\{\DK^2,\DO^2+d_t\}}\right)^2,
		\end{align*}
		we can use Lemma~\ref{lem: stoc one iter improve} with $d_t=\frac{8\beta\DK^2}{\alpha(t+1)}$, which, by our choice of step size $\eta_t$ implies that,
		\begin{align*}
			\forall t\geq 1:\quad  \hbar_{t+1}\leq (1-\frac{2}{t+1})\hbar_t+ \frac{4\beta}{(t+1)^2}\big(\DO^2+\frac{8\beta\DK^2}{\alpha(t+1)}\big).
		\end{align*}
		Thus, by Lemma~\ref{lem: new recursion}, for any $t\geq 2$, we have that,
		\begin{align*}
			\hbar_{t}&\leq \frac{1}{t^2}\sum_{k=1}^{t-1} 4\beta\big(\DO^2+\frac{8\beta\DK^2}{\alpha(k+1)}\big)\leq
			\frac{\frac{1}{t-1}\sum_{k=1}^{t-1}4\beta\DO^2}{t+1}+\frac{32\beta^2\DK^2}{\alpha}\cdot\frac{\sum_{k=1}^{t}\frac{1}{k}-1}{t^2}\\&\leq
			\frac{4\beta\DO^2}{t+1}+\frac{\frac{32\beta^2}{\alpha}\DK^2\log(t)}{t^2},
		\end{align*}
		concluding the proof.
	\end{proof}
	
	\subsection{Proof of Theorem~\ref{thm:stoc eps convergence}}\label{sec:stochasticComplexProof}
	For clarity, we first restate the theorem and then prove it.
	\begin{reptheorem}{thm:stoc eps convergence}
		Suppose $f,\mK$ satisfy the quadratic growth property with some $\alpha >0$. Using Algorithm~\ref{alg:Algorithm stochastic} with step-size $\eta_t=\frac{2}{t+1}$ and mini-batch sizes that satisfy 
		\begin{align*}
			m_t= \max\Big\{\Big(\frac{G(t+1)}{\beta\DK}\Big)^2,\min\bigg\{\Big(\frac{G\DK(t+1)}{\beta\DO^2}\Big)^2,\Big(\frac{\alpha G(t+1)^2}{8\beta^2 \DK}\Big)^2\Big\}\Big\},
		\end{align*} 
		for any $0<\epsilon\leq \frac{16\beta^2\DK^2}{\alpha}$,  expected approximation error $\epsilon$ is achieved after
		$\tilO\big(\max\{\frac{\beta\DO^2}{\epsilon}, \frac{\beta\DK}{\sqrt{\alpha\epsilon}}\}\big)$
		calls to the NEP oracle and
		$\tilO\big(\beta G^2\max\big\{\frac{\DK^2\DO^2}{\epsilon^3}, \frac{\DK}{\alpha^{3/2}\epsilon^{3/2}},\frac{\DK^3}{\alpha^{1/2}\epsilon^{5/2}}\big\}\big)$
		stochastic gradient evaluations, where $\tilO$ suppresses poly-logarithmic terms in $\frac{\beta^2\DK^2}{\alpha\epsilon}$. 
	\end{reptheorem}
	\begin{proof}
		Let $0<\epsilon\leq \frac{16\beta^2\DK^2}{\alpha}$ and note that for $t\geq 8\beta\DO^2\epsilon^{-1}$, it holds that $\frac{4\beta\DO^2}{t+1}\leq \frac{\epsilon}{2}$.\\
		Now, let $M>0$ such that $\sqrt{\frac{2M}{\epsilon}\log(\frac{2M}{\epsilon})}\geq 2$, and observe that for all $t\geq \sqrt{\frac{2M}{\epsilon}\log(\frac{2M}{\epsilon})}$ it holds that,
		\begin{align*}
			\frac{M\log(t)}{t^2}\underset{(a)}\leq \frac{M\log(\sqrt{\frac{2M}{\epsilon}\log(\frac{2M}{\epsilon})})}{\frac{2M}{\epsilon}\log(\frac{2M}{\epsilon})}\underset{(b)}\leq\frac{M\log(\frac{2M}{\epsilon})}{\frac{2M}{\epsilon}\log(\frac{2M}{\epsilon})}=\frac{\epsilon}{2},
		\end{align*}
		where (a) holds since $\frac{M\log(t)}{t^2}$ is monotonically decreasing in $t$ for $t\geq 2$, and (b) holds since $\log(\frac{2M}{\epsilon})\leq \frac{2M}{\epsilon}$. 
		
		Thus, taking $M=\frac{32\beta^2\DK^2}{\alpha}$, note that since $0<\epsilon\leq \frac{16\beta^2\DK^2}{\alpha}=\frac{M}{2}$, we have that, $\sqrt{\frac{2M}{\epsilon}\log(\frac{2M}{\epsilon})}\geq \sqrt{4\log(4)}>2$, and thus, we have that,
		\begin{align*}
			\forall t\geq \sqrt{\frac{64\beta^2 D_{\mK}^2}{\alpha\epsilon}\log\left(\frac{64\beta^2D_{\mK}^2}{\alpha\epsilon}\right)}: \quad \frac{\frac{32\beta^2}{\alpha}\DK^2\log(t)}{t^2}\leq \frac{\epsilon}{2}.
		\end{align*}
		Thus, denoting $T=\Big\lceil\max\left\{\frac{8\beta\DO^2}{\epsilon}, \sqrt{\frac{64\beta^2 D_{\mK}^2}{\alpha\epsilon}\log\left(\frac{64\beta^2D_{\mK}^2}{\alpha\epsilon}\right)}\right\}\Big\rceil$, using Theorem~\ref{thm: stoc improved rate} we have that for all $t\geq T$,
		\begin{align*}
			\hbar_t\leq \frac{4\beta\DO^2}{t+1}+\frac{\frac{32\beta^2}{\alpha}\DK^2\log(t)}{t^2}\leq \frac{\epsilon}{2}+\frac{\epsilon}{2} = \epsilon.
		\end{align*}
		Thus, we indeed reach $\epsilon$ expected approximation error in $O\big(\max\{\frac{\beta\DO^2}{\epsilon}, \frac{\beta\DK\log^{1/2}(\frac{\beta^2\DK^2}{\alpha\epsilon})}{\sqrt{\alpha\epsilon}}\}\big)$ calls to the linear oracle.
		
		Now, let $n_{grad}$ be the number of stochastic gradient used until iteration $T$. Note that we have that,
		\begin{align}
			n_{grad}&=\sum_{t=1}^{T-1} m_t \leq Tm_{T-1} = T\max\bigg\{\left(\frac{GT}{\beta\DK}\right)^2,\min\bigg\{\left(\frac{G\DK T}{\beta\DO^2}\right)^2,\left(\frac{\alpha GT^2}{8\beta^2 \DK}\right)^2\bigg\}\bigg\}\nonumber\\&=
			\max\bigg\{\frac{G^2T^3}{\beta^2\DK^2},\min\bigg\{\frac{G^2\DK^2 T^3}{\beta^2\DO^4},\frac{\alpha^2 G^2T^5}{64\beta^4 \DK^2}\bigg\}\bigg\}.\label{eq:thm7 smp_bound}
		\end{align}
		Thus, when $\frac{8\beta\DO^2}{\epsilon}\geq \sqrt{\frac{64\beta^2 D_{\mK}^2}{\alpha\epsilon}\log(\frac{64\beta^2D_{\mK}^2}{\alpha\epsilon})}$ we have that, $T=\big\lceil\frac{8\beta\DO^2}{\epsilon}\big\rceil=O(\frac{\beta\DO^2}{\epsilon})$, and thus using \eqref{eq:thm7 smp_bound}, we have that,
		\begin{align*}
			n_{grad}&\leq \max\bigg\{\frac{G^2T^3}{\beta^2\DK^2},\frac{G^2\DK^2 T^3}{\beta^2\DO^4}\bigg\}\underset{(a)}= \frac{G^2\DK^2 T^3}{\beta^2\DO^4} = O\left(\frac{G^2\DK^2}{\beta^2\DO^4}\cdot \frac{\beta^3\DO^6}{\epsilon^3}\right)\\&=O\left(\frac{\beta G^2\DK^2\DO^2}{\epsilon^3}\right),
		\end{align*}
		where (a) holds since $\DO\leq\DK$.
		
		Otherwise, we have that, $T=\left\lceil\sqrt{\frac{64\beta^2 D_{\mK}^2}{\alpha\epsilon}\log(\frac{64\beta^2D_{\mK}^2}{\alpha\epsilon})}\right\rceil=O\big(\frac{\beta\DK\log^{1/2}(\frac{\beta^2\DK^2}{\alpha\epsilon})}{\sqrt{\alpha\epsilon}}\big)$, and thus, using \eqref{eq:thm7 smp_bound} we have that,
		\begin{align*}
			n_{grad}&\leq \max\bigg\{\frac{G^2T^3}{\beta^2\DK^2},\frac{\alpha^2 G^2T^5}{64\beta^4 \DK^2}\bigg\}\\&=O\left(\max\bigg\{\frac{G^2}{\beta^2\DK^2}\cdot \frac{\beta^3\DK^3\log^{3/2}(\frac{\beta^2\DK^2}{\alpha\epsilon})}{\alpha^{3/2}\epsilon^{3/2}},\frac{\alpha^2 G^2}{\beta^4 \DK^2}\cdot\frac{\beta^5\DK^5\log^{5/2}(\frac{\beta^2\DK^2}{\alpha\epsilon})}{\alpha^{5/2}\epsilon^{5/2}}\bigg\}\right)\\&=
			O\left(\max\bigg\{\frac{\beta G^2\DK\log^{3/2}(\frac{\beta^2\DK^2}{\alpha\epsilon})}{\alpha^{3/2}\epsilon^{3/2}},\frac{\beta G^2\DK^3\log^{5/2}(\frac{\beta^2\DK^2}{\alpha\epsilon})}{\alpha^{1/2}\epsilon^{5/2}}\bigg\}\right).
		\end{align*}
		
		Thus, we indeed achieve $\epsilon$ expected approximation error after
		\begin{align*}
			O\left(\beta G^2\max\bigg\{\frac{\DK^2\DO^2}{\epsilon^3},\frac{\DK\log^{3/2}(\frac{\beta^2\DK^2}{\alpha\epsilon})}{\alpha^{3/2}\epsilon^{3/2}},\frac{\DK^3\log^{5/2}(\frac{\beta^2\DK^2}{\alpha\epsilon})}{\alpha^{1/2}\epsilon^{5/2}}\bigg\}\right)
		\end{align*}
		stochastic gradient evaluations.
	\end{proof}
	
	\section{Additional numerical results}\label{sec: apendexpriments}
	\subsection{Hypercube-constrained least-squares}

	\begin{figure}[H]
		\centering
		\begin{minipage}[b]{0.49\textwidth}
			\includegraphics[width=\textwidth]{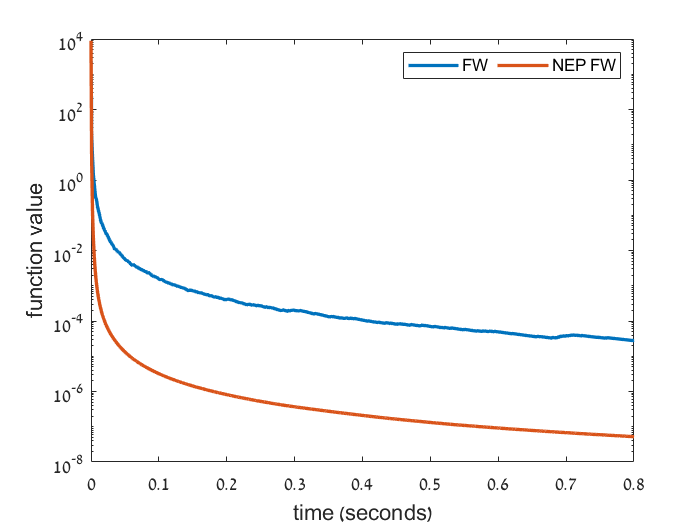}
		\end{minipage}
		\caption{Comparison between the standard Frank-Wolfe method (FW) and our NEP FW variant on the hypercube-constrained least-squares problem. These are the same results as in Figure \ref{fig: cubels} but focusing only on these two variants. }
		\label{fig: cubefaircomp1}
	\end{figure}
	
	\begin{figure}[H]
		\centering
		\begin{minipage}[b]{0.49\textwidth}
			\includegraphics[width=\textwidth]{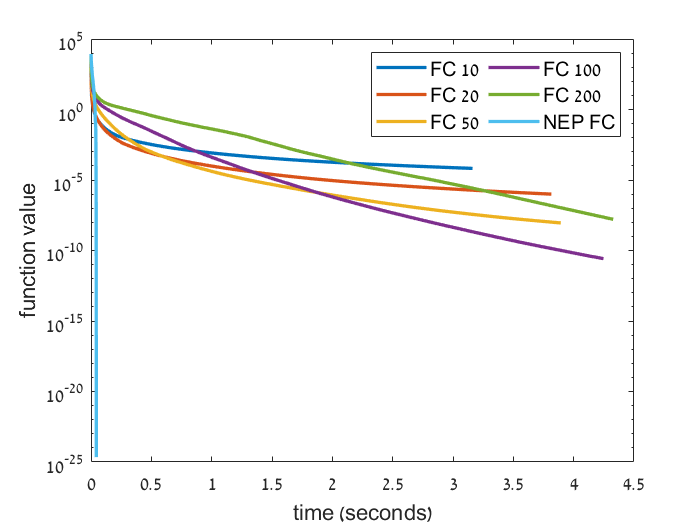}
		\end{minipage}
		\caption{Comparison between our NEP FC variant with 50 iterations of FISTA per iteration and the FC variant with various choices for the number of FISTA iterations (the number of  FISTA iterations is shown besides the algorithm's name) on the hypercube-constrained least-squares problem. The results are the average of 50 i.i.d. runs (where in each run we sample fresh $\A,\x^*$).}
		\label{fig: cubefaircomp2}
	\end{figure}

	\subsection{Video co-localization}
	\begin{figure}[h!]
		\centering
		\begin{minipage}{0.49\textwidth}
			\includegraphics[width=\textwidth]{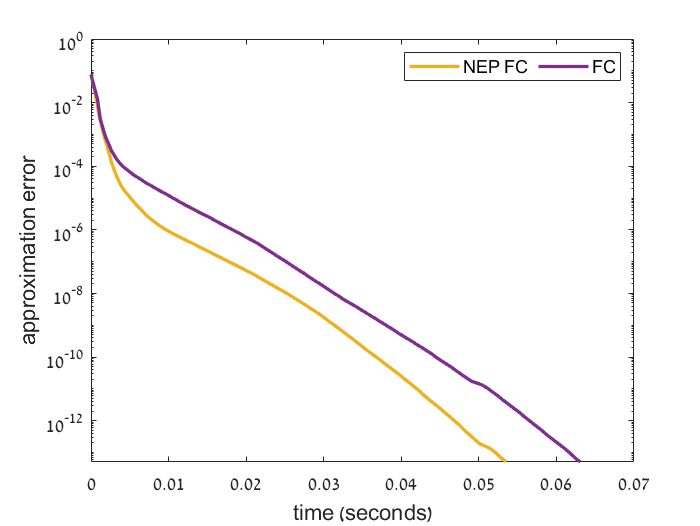}
		\end{minipage}
		\begin{minipage}{0.49\textwidth}
			\includegraphics[width=\textwidth]{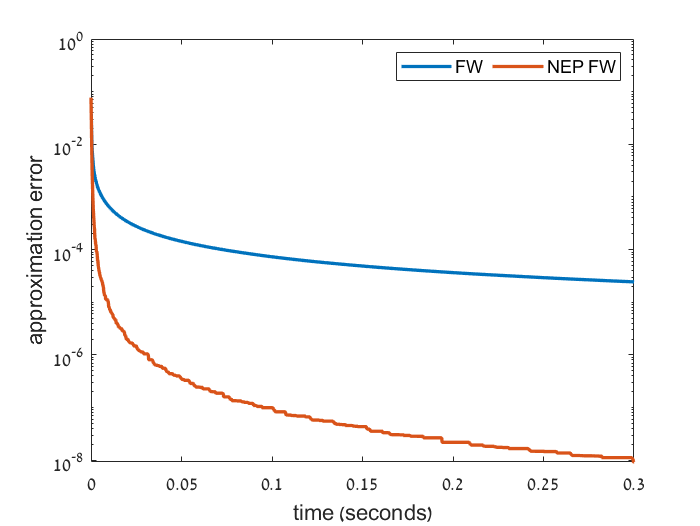}
		\end{minipage}
		\caption{Comparison of Frank-Wolfe variants on the video co-localization problem. The times shown are the averages of 200 runs. These are the same results as in Figure \ref{fig: videoco}, but focusing on FC vs. NEP FC (left panel) and FW vs. NEP FW (right panel). }
		\label{fig: videocoapend}
	\end{figure}
	
	In Figure~\ref{fig: videocoapendgap} we present the performance of the algorithms measured by the duality gap (as was done in \cite{lacoste2015global} and \cite{GM16}). It can be seen that although NEP FC still outperforms all other algorithms, it only gives a slight improvement over FC and that NEP FW only slightly outperforms FW, and is out preformed by AFW and PFW both with respect to time and number of iterations. Here we remind the reader that while we proved the theoretical superiority of our NEP oracle-based algorithms w.r.t. the (primal) approximation error, we did not give any improved bounds w.r.t. the duality gap, and we leave it for future work to settle the question whether or not the use of a NEP oracle could lead to provably faster dual convergence.
	\begin{figure}[h!]
		\centering
		\begin{minipage}[b]{0.49\textwidth}
			\includegraphics[width=\textwidth]{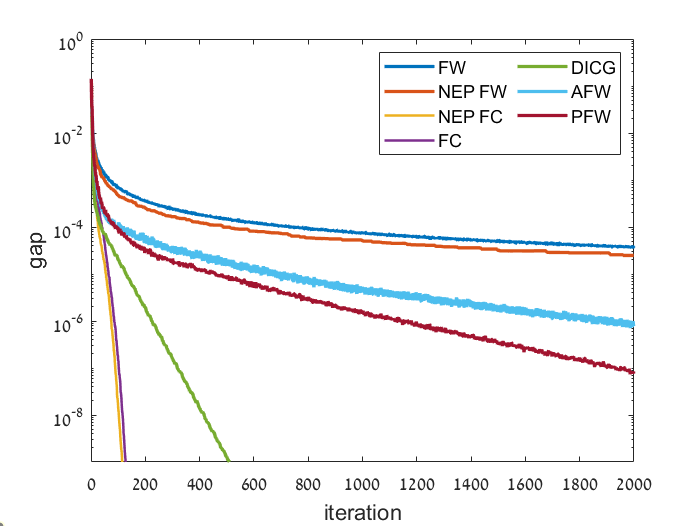}
		\end{minipage}
		\begin{minipage}[b]{0.49\textwidth}
			\includegraphics[width=\textwidth]{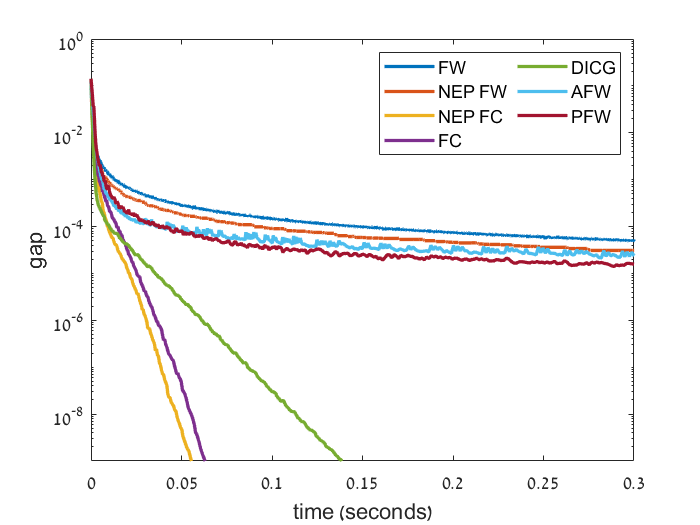}
		\end{minipage}
		\caption{Comparison of Frank-Wolfe variants on the video co-localization problem in terms of the duality gap.  The times shown are the averages of 200 runs.}
		\label{fig: videocoapendgap}
	\end{figure}
\end{document}